\documentclass[12pt]{amsart}

  \usepackage{a4wide}
  \usepackage{amssymb}
  \usepackage{amsmath}
  \usepackage{amsthm}
  \usepackage{amsfonts}
  \usepackage{amscd}
  \usepackage{boxedminipage}
  \usepackage{nomencl}
  \usepackage{verbatim}
  \usepackage{booktabs}
  \usepackage{lscape}
  \usepackage[figuresright]{rotating}
    \usepackage[T1]{fontenc}
  \usepackage[usenames]{color}

  \usepackage{pstricks,pst-node,pst-tree,pst-poly}
  \psset{unit=\unitlength}
  \psset{griddots=6,gridlabels=0,gridcolor=gray,subgriddiv=1}
  \newpsobject{showgrid}{psgrid}{subgriddiv=2,griddots=10,gridlabels=6pt,gridcolor=lightgray}
  \newif\ifhrule\hrulefalse

  \makeatletter
  \CheckCommand*\refstepcounter[1]{\stepcounter{#1}%
      \protected@edef\@currentlabel
       {\csname p@#1\endcsname\csname the#1\endcsname}%
  }
  \renewcommand*\refstepcounter[1]{\stepcounter{#1}%
    \protected@edef\@currentlabel
      {\csname p@#1\expandafter\endcsname\csname the#1\endcsname}%
  }
  \def\labelformat#1{\expandafter\def\csname p@#1\endcsname##1}
  \DeclareRobustCommand\Ref[1]{\protected@edef\@tempa{\ref{#1}}%
     \expandafter\MakeUppercase\@tempa
  }
  \makeatother

  \makeatletter
  \newcommand{\numberlike}[2]{%
     \expandafter\def\csname c@#1\endcsname{%
         \expandafter\csname c@#2\endcsname}%
  }
  \makeatother

  \def\DefaultNumberTheoremWithin{section}

  \theoremstyle{plain}
  \newtheorem{Lemma}{Lemma}
     \numberwithin{Lemma}{\DefaultNumberTheoremWithin}
     \labelformat{Lemma}{Lemma~#1}
  
     \numberwithin{Claim}{\DefaultNumberTheoremWithin}
     \numberlike{Claim}{Lemma}
     \labelformat{Claim}{Claim~#1}
  \newtheorem{Theorem}{Theorem}
     \numberwithin{Theorem}{\DefaultNumberTheoremWithin}
     \numberlike{Theorem}{Lemma}
     \labelformat{Theorem}{Theorem~#1}
  \newtheorem{Corollary}{Corollary}
     \numberwithin{Corollary}{\DefaultNumberTheoremWithin}
     \numberlike{Corollary}{Lemma}
     \labelformat{Corollary}{Corollary~#1}
  \newtheorem{Proposition}{Proposition}
     \numberwithin{Proposition}{\DefaultNumberTheoremWithin}
     \numberlike{Proposition}{Lemma}
     \labelformat{Proposition}{Proposition~#1}
  
     \numberwithin{Conjecture}{\DefaultNumberTheoremWithin}
     \numberlike{Conjecture}{Lemma}
     \labelformat{Conjecture}{Conjecture~#1}

  \theoremstyle{definition}
  
     \numberwithin{Definition}{\DefaultNumberTheoremWithin}
     \numberlike{Definition}{Lemma}
     \labelformat{Definition}{Definition~#1}

  \theoremstyle{definition}
  
     \numberwithin{Question}{\DefaultNumberTheoremWithin}
     \numberlike{Question}{Lemma}
     \labelformat{Question}{Question~#1}

  \theoremstyle{definition}
  
     \numberwithin{Problem}{\DefaultNumberTheoremWithin}
     \numberlike{Problem}{Lemma}
     \labelformat{Problem}{Problem~#1}

  \theoremstyle{remark}
  \newtheorem{Remark}{Remark}
     \numberwithin{Remark}{\DefaultNumberTheoremWithin}
     \numberlike{Remark}{Lemma}
     \labelformat{Remark}{Remark~#1}
  \theoremstyle{remark}

  \newtheorem{Example}{Example}
     \numberwithin{Example}{\DefaultNumberTheoremWithin}
     \numberlike{Example}{Lemma}
     \labelformat{Example}{Example~#1}
  
     \labelformat{Case}{Case~#1}
     \numberwithin{Case}{Lemma}
  
     \labelformat{Step}{Step~#1}
     \numberwithin{Step}{Lemma}

  \labelformat{equation}{(#1)}
  \labelformat{figure}{Figure~#1}
  \labelformat{section}{\S#1}
  \labelformat{subsection}{\S#1}
  \labelformat{subsubsection}{\S#1}

  \def\eqref{\ref}

  \allowdisplaybreaks
  \newcommand{\mb}{\mathbb}
  \newcommand{\mc}{\mathcal}
  \newcommand{\aA}{\mathcal A}
  \newcommand{\mf}{\mathfrak}

  \def\binomial(#1,#2){{#1\choose #2}}
  \def\delplus(#1,#2,#3){\ensuremath{\pi}_{#1,#2,#3}}
  \def\delminus(#1,#2){\ensuremath{\operatorname{\partial}}_{#1,#2}}
  \newcommand{\Model}{\mathcal{M}}
  
  \def\redword(#1){{\ensuremath{\tilde{#1}}}}

  \newcommand{\LLL}{\mc{L}}
  \newcommand{\OOO}{\mc{O}}

  \newcommand{\ZZ}{\mb{Z}}
  \newcommand{\NN}{\mb{N}}
  \newcommand{\PP}{\mb{P}}
  
  \newcommand{\RR}{\mb{R}}
  \newcommand{\CC}{\mb{C}}
  \newcommand{\KK}{\mb{K}}

  \newcommand{\xx}{\mathbf{x}}

  \newcommand{\symm}{\mf{S}}

  \newcommand{\Bottom}{ascending~}
  \newcommand{\bott}{\mathrm{asc}}
  \newcommand{\BottomPol}{\ensuremath{\operatorname{As}}}
  \newcommand{\Csiszar}{Csisz\'ar~}
  \newcommand{\csi}{\mathrm{csi}}
  \newcommand{\CsiszarPol}{\ensuremath{\operatorname{Ci}}}
  \newcommand{\Inversion}{inversion~}
  \newcommand{\inv}{\mathrm{inv}}
  \newcommand{\Birkhoff}{Birkhoff~}
  \newcommand{\birk}{\mathrm{birk}}
  \newcommand{\BirkhoffPol}{\ensuremath{\operatorname{Bi}}}
  \newcommand{\PlackettLuce}{Plackett-Luce~}
  \newcommand{\PlackettLuceShort}{\mathrm{PL}}
  \newcommand{\pl}{\mathrm{PL}}
  \newcommand{\BradleyTerry}{Bradley-Terry~}
  \newcommand{\bt}{\mathrm{BT}}

  \newcommand{\GeneralPoset}{Q}
  \newcommand{\PartialRanking}{\mathcal{P}}
  \newcommand{\LinExt}{\LLL}
  \newcommand{\GroundSet}{\Omega}
  \newcommand{\MaxChain}{\mathrm{M}}
  \newcommand{\Cov}{\mathrm{Cov}}
  \newcommand{\MinCov}{\mathrm{MinCov}}
  \newcommand{\Order}{\OOO}
  \newcommand{\rk}{\mathrm{rk}}

  \newcommand{\object}{object~}

  \newenvironment{note}[1][Note]
   {\bigskip\begin{center}\begin{boxedminipage}{4.5in}\setlength{\parindent}{1em}\noindent\textbf{#1. }}
   {\end{boxedminipage}\end{center}\bigskip}

  \makenomenclature

  \makeindex

\begin{document}

  \title{Commutative Algebra of Statistical Ranking}

  \author{Bernd Sturmfels}
     \address{              Department of Mathematics \\
              University of California \\
              Berkeley, CA 94720 \\
              USA}
     \email{bernd@math.berkeley.edu}

  \author{Volkmar Welker}
     \address{Fachbereich Mathematik und Informatik\\
              Philipps-Universit\"at \\
              35032 Marburg\\
              Germany}
     \email{welker@mathematik.uni-marburg.de}

  \begin{abstract}
    A model for statistical ranking is a family of probability
    distributions whose states are orderings of a fixed finite set of items.
    We represent the orderings as
        maximal chains in a graded poset. The most widely used
            ranking models are parameterized by 
     rational function in the model
     parameters, so they define algebraic varieties.
     We study    these varieties from the perspective
    of combinatorial commutative algebra. 
    One of our models, the Plackett-Luce model, is non-toric.
    Five others are toric:
    the Birkhoff model, the
    ascending model, the Csisz\'ar model,
    the inversion model, and the Bradley-Terry model.
    For these models we examine the toric algebra,
    its lattice polytope, and its Markov basis.
  \end{abstract}

  \maketitle
  
  \markboth{Bernd Sturmfels and Volkmar Welker}{
  Commutative Algebra of Statistical Ranking}
  
 \section{Introduction}
 
  A statistical model for ranked data is a family $\Model$ of probability distribution on the
  symmetric group $\symm_n$. Each distribution $p(\theta)$ in 
  $\Model$ depends on some model
  parameters $\theta$ and it associates a probability $p_\pi(\theta)$ to each
  permutation $\pi$ of $[n] = \{1,2,\ldots,n\}$.
  Thus the model $\Model$ is a parametrized subset of the 
  $(n!-1)$-dimensional standard simplex $\Delta_{\symm_n}$.

  In algebraic statistics, one assumes that the probabilities $p_\pi(\theta)$
  are rational functions in the model parameters $\theta$, so that $\Model$
  is a semi-algebraic set in $ \Delta_{\symm_n}$, and one aims to
  characterize the prime ideal $I_\Model$ of polynomials that vanish on $\Model$.
  In fact, one of the origins of the field was the
  {\em spectral analysis for permutation data} described by
  Diaconis and Sturmfels in \cite[\S 6.1]{DiaconisSturmfels1998}.
  The corresponding {\em Birkhoff model} $\Model$
     is the toric variety of the Birkhoff polytope. This polytope consists of all
  bistochastic matrices and it is the convex hull of all $n {\times} n$ permutation matrices.
  There has been a considerable amount of research on  the
  geometric invariants  of the Birkhoff model $\Model$.
  The simplest such invariant is its dimension, ${\rm dim}(\Model) = (n-1)^2$.
  The {\em degree} of $\Model$ is the
  normalized volume of the Birkhoff polytope, a topic of independent interest
  in  combinatorics \cite{CanfieldMcKay2009}.
  Diaconis and Eriksson \cite{DiaconisEriksson2006} conjectured that the Markov basis of the Birkhoff 
  model consists of binomials of degree~$\leq 3$.
  
  Besides the Birkhoff model, there are many other models for ranked data that 
  are both relevant for statistical analysis and have an interesting algebraic structure.
  It is the objective of this article to conduct a comparative study of such models
  from the perspectives of commutative algebra and geometric combinatorics.
  Both toric models and non-toric models are of interest. The former include the models introduced by
  Csisz\'ar \cite{Csiszar2009a,Csiszar2009b}, and the latter include
  the Plackett-Luce model \cite{ChengDembczynskiHuellermeier2010, Luce1959, Plackett1968} and
  the generalized \BradleyTerry models  \cite{Hunter2004}.

  \smallskip

  The organization of this paper is as follows.
  In Section 2 we give an informal introduction to all our models.
  We write out formulas for the probabilities
  for the six permutations of  $n = 3$ items, and we discuss the subsets they
  parametrize in
  the $5$-dimensional simplex $\Delta_{\symm_3}$.
  Precise formal definitions for the four toric models are given in Section 3.
  We represent the states as maximal chains in a graded poset $Q$.
  Typically, $Q$ is the distributive lattice induced by some order 
  constraints on the $n$ items to be ranked. If there are no such constraints then
  $Q = 2^{[n]}$ is the Boolean lattice whose maximal chains
  are all $n !$ permutations in $\symm_n$.  
  Non-trivial order constraints arise frequently in applications   of ranking models, for instance in
  computational biology \cite{BeerenwinkelErikssonSturmfels2006}
  and machine learning \cite{ChengDembczynskiHuellermeier2010}.
  Our algebraic framework based on graded posets $Q$ 
  is well-suited for such contemporary  applications of  statistical ranking.

  While the Birkhoff model has already received a lot of attention in the literature,
  we here focus on the {\em Csisz\'ar model} (Section 4), the
  {\em \Bottom model} (Section 5) and the
  {\em \Inversion model} (Section 6). For each of these toric
  varieties, we characterize the corresponding lattice polytope
  and its Markov bases, that is, binomials that generate the toric ideal.
 
  Section 7 is concerned with the {\em Plackett-Luce model}, which is not
  a toric model, but is parametrized by
  certain conditional probabilities that are not monomials. 
  In algebraic geometry language, this model is obtained 
  by blowing up the projective space
  $\PP^{n-1}$ along a  family of linear subspaces of codimension $2$,
  and we study its
  coordinate ring.
  We also examine marginalizations of our models,  including
  the widely used {\em Bradley-Terry model}.
  
 \section{Toric Models: A Sneak Preview}
  \label{sec:toricfirst}
  
  A toric model for complete permutation data is specified by a non-negative
  integer matrix $A$ with $n!$ columns that all have the same sum $S$.
  These column vectors $A_\pi$ are indexed by permutations $\pi \in \symm_n$ and they
  represent the {\em sufficient statistics} of the model.
  The article \cite{GeigerMeekSturmfels2006} serves as our general
  reference for toric models in statistics, their relationship with exponential families,
  and the role of the matrix $A$. For an introduction to algebraic statistics
  in general, and for further reading on toric models,
  we refer to the books \cite{DrtonSturmfelsSullivant2009, PachterSturmfels2005}.

  If $r = {\rm rank}(A)$ then the convex hull of the column vectors 
  $A_\pi$ is a lattice polytope of dimension $r-1$. 
  We refer to it as the {\em model polytope}.   The toric model  
  can be identified with the  non-negative points on the
  projective toric variety associated with the model polytope.
       Each  data set is summarized as
  a function $u : \symm_n \mapsto \NN$,
  where $u(\pi)$ is the number of times the permutation $\pi$ has been observed.
  Thinking of $u$ as a column vector, we can form the matrix-vector product $A u$,
  whose entries are the sufficient statistics of the data $u$.
  Then the sum $n! \cdot S$ of the entries in the vector $Au$ coincides with 
  the sample size $\,N = \sum_{\pi \in \symm_n} u(\pi)$.

  In subsequent sections we will generalize to the situation 
  where $\symm_n$ is replaced by a proper subset,
  in which case $A$ has fewer than $n!$ columns, but
  still labeled by permutations. These will be the 
  linear extensions of a given partial order on $[n] = \{1,2,\ldots,n\}$.
  In fact, for some models we can even take the
  set of maximal chains in an arbitrary ranked poset. But for a first look we confine
  ourselves to the situation described above, where $A$ has $n !$ columns.
  
  We now define four toric models for probability distributions on $\symm_n$. 
  We do this by way of
  a verbal description of the sufficient statistics in each model.
  These sufficient statistics are   numerical functions on
  the permutations $\pi$ of the given set $[n]$ of items to be ranked.
  
  \begin{itemize}
    \item[(a)] In the {\em \Bottom model}, the sufficient statistics $Au$ record,
      for each subset $I \subset [n]$, the number of samples $\pi$ in the data $u$
      that have the set $I$ at the bottom. Here, the set
      $I$ being {\em at the bottom} means that
            ($i \in I$ and $j \not\in I$) implies $\pi(i) < \pi(j)$.
    \item[(b)] In the {\em \Csiszar model}, the sufficient statistics $Au$ count,
      for each $i \in I \subset [n]$, the number of samples
      that have $I$ at the bottom but with $i$ as winner in the group~$I$.
      This is the model studied by Vill\~ o \Csiszar  \cite{Csiszar2009a, Csiszar2009b}
       under the name ``L-decomposable''.
    \item[(c)] In the {\em \Birkhoff model} of \cite[\S 6.1]{DiaconisSturmfels1998},
      the sufficient statistics $Au$ of a data set $u$ record, for each $i,j \in [n]$,
      the number of samples $\pi$ in which object $i$ is ranked in place $j$,
    \item[(d)] In the {\em \Inversion model}, 
      the sufficient statistics $Au$ count, for each ordered pair $i<j$ in $ [n]$,
      the number of samples $\pi$ in which that pair is 
      an inversion, meaning $\pi^ {-1}(i) > \pi^{-1}(j)$.
      This model can be seen as a multivariate version of the {\em Mallows model} 
      \cite{Marden1995}.
    
  \end{itemize}
  
  \smallskip

  To illustrate the differences between these models let us consider the simplest case $n=3$.
  In each case the toric ideal of the model is the kernel 
  of a square-free monomial map from the polynomial ring
  $\KK[p_{123},p_{132},p_{213}, p_{231},p_{312},p_{321}]$ representing the probabilities
  to another polynomial ring $\KK[a,b, \ldots\,] $ that represents the model parameters.
  The model polytope is the convex hull of the six $0$-$1$ vectors 
  corresponding to the square-free monomials:
  $$ 
    \begin{matrix}
                     &  p_{123}  & p_{132} & p_{213} & p_{231} & p_{312} & p_{321} \\
      {\mbox{\Birkhoff}} &  a_{11} a_{22} a_{33} &  a_{11} a_{23} a_{32} & a_{12} a_{21} a_{33} & 
                        a_{12} a_{23} a_{31} & a_{13} a_{21} a_{32} & a_{13} a_{22} a_{31} \\
      {\mbox{\Inversion}} & b_{12} b_{13} b_{23} &
                        b_{12} b_{13} q_{23} & q_{12} b_{13} b_{23} &
                        q_{12} q_{13} b_{23} &b_{12} q_{13} q_{23} & q_{12} q_{13} q_{23} \\
      {\mbox{\Bottom}} & c_1 c_{12} c_{123} & c_1 c_{13} c_{123} & c_2  c_{12} c_{123} & c_2  c_{23} c_{123} 
                  & c_3 c_{13} c_{123} & c_3 c_{23} c_{123} \\ 
      {\mbox{\Csiszar}} & 
      d_{|1} d_{1|2} d_{12|3} &
      d_{|1} d_{1|3} d_{13|2} &
      d_{|2} d_{2|1} d_{12|3} &
      d_{|2} d_{2|3} d_{23|1} &
      d_{|3} d_{3|1} d_{13|2} &
      d_{|3} d_{3|2} d_{23|1} \\
    \end{matrix}
  $$
  The toric ideals record the algebraic relations among these square-free monomials:
  $$
    \begin{matrix}
     % I_{\rm \birk} & = &  \langle \, p_{123} p_{231} p_{312} - p_{132} p_{213} p_{321} \, \rangle &
       %   & \text{has codimension $1$,} \\
      I_{\inv}  & = & \langle \,
          p_{132}p_{231} - p_{123} p_{321} , \,
          p_{213} p_{312} - p_{123} p_{321} \,
          \rangle &
          & \text{has codimension $2$,} \\
           I_{\rm \birk}   = I_{\bott} & = & \langle \, p_{123} p_{231} p_{312} - p_{132} p_{213} p_{321} \,  \rangle &
          & \text{has codimension $1$,} \\
      I_{\csi} & = & \langle \,\, 0 \,\, \rangle &
          & \text{has codimension $0$.} \\
    \end{matrix}
  $$
  
  For each model, the matrix $A$ has six columns, indexed by
  $\symm_3$, and its rows are labeled by the model parameters.
  For example, for the ascending model, the matrix has seven rows:
  $$  
    \BottomPol_3 \quad = \quad
    \bordermatrix{
          & \,p_{123} & p_{132} & p_{213} & p_{231} & p_{312} & p_{321} \cr
      c_1 &         1 &       1 &       0 &       0 &       0 &       0 \cr
      c_2 &         0 &       0 &       1 &       1 &       0 &       0 \cr
      c_3 &         0 &       0 &       0 &       0 &       1 &       1 \cr
   c_{12} &         1 &       0 &       1 &       0 &       0 &       0 \cr
   c_{13} &         0 &       1 &       0 &       0 &       1 &       0 \cr
   c_{23} &         0 &       0 &       0 &       1 &       0 &       1 \cr
   c_{123}&         1 &       1 &       1 &       1 &       1 &       1
    } 
  $$
  Here we use the same notation for both the
  matrix and the model polytope, which is the convex hull
  of the columns. From the equality of ideals,
  $I_{\birk} = I_{\bott}$, we infer that the polytope  $\BottomPol_3$ is affinely isomorphic to
  the $3 {\times} 3$-\Birkhoff polytope, which is a
  cyclic $4$-polytope with six vertices.
  The ideal $I_\inv$ reveals that
  the model polytope for the \Inversion model is a {\em regular octahedron},
  while the polytope for the \Csiszar model is the full {\em $5$-simplex}.

  To see that no two of our four models agree,
  we need to go to $n \geq 4$. 
   
  \begin{Example}
     \label{ex:strict}
    Let  $n=4$. 
    Then all four model polytopes have $24$ vertices but their dimensions 
    are different. 
    The \Birkhoff model has dimension  $9$, the inversion model has dimension $6$, 
    the ascending model has dimension $11$, and     the  
    Csisz\' ar model has dimension $17$.
     \ref{thm:inclusions} will explain 
  the precise relationships and inclusions among the four models. \qed
   \end{Example}

  Our work on this project started by trying to understand a
  certain model whose toric closure  is the ascending model.
  Here {\em toric closure} refers to the smallest
  toric model containing a given model. That non-toric model
  for ranking  is  the {\em Plackett-Luce model}
  \cite{ChengDembczynskiHuellermeier2010, Luce1959, Plackett1968}.
  It can be obtained from
  the ascending model by the following specialization of parameters:
  $$ 
     c_i \mapsto \frac{1}{\theta_i}, \,\,
     c_{ij} \mapsto \frac{1}{\theta_i + \theta_j},\,\,
     c_{ijk} \mapsto \frac{1}{\theta_i + \theta_j + \theta_k} ,\,\ldots . 
  $$
  The prime ideal of algebraic relations among the $p_\pi$ is
  a non-toric ideal which contains the toric ideal $I_\bott$.
  The case $n=3$ is worked out explicitly in \ref{exPL3}.
  Geometrically, that smallest Plackett-Luce model 
  corresponds to blowing up $\,\PP^2\,$ at the nine points in
   \ref{eq:ninepoints}.

 \section{Toric Models: Definitions and General Results}

  Let $\GeneralPoset$ be a poset on finite ground set $\GroundSet$. 
  A {\em $\GeneralPoset$-ranking} is a
  maximal chain $a_0 < \cdots < a_n$ in $\GeneralPoset$.
  A chain $a_0 < \cdots < a_n$ being {\em maximal}
    means that  $a_0$ is minimal in $\GeneralPoset$, $a_n$ is maximal,
  and $a_i < a_{i+1}$ is a cover relation for $0 \leq i \leq n-1$.
  We write $\MaxChain(\GeneralPoset)$ for the set of maximal chains in $\GeneralPoset$ and
  $\Cov(\GeneralPoset)$ for the set of cover relations in $\GeneralPoset$.
  If $\GeneralPoset = 2^{[n]}$ is the Boolean lattice of all subsets of $[n]$ ordered by
  inclusion then the maximal chains in $\GeneralPoset$ are in bijection with the 
  permutations in $\symm_n$, and the models below coincide with the
  ones described in Section~2.

We shall define four toric models whose
states are the maximal chains 
$ \pi \in \MaxChain(\GeneralPoset)$.
The probability of $\pi$ is represented by an indeterminate $p_\pi$.
   Each toric model for $Q$-rankings is
  defined by a non-negative integer matrix $A$ whose columns are
  indexed by  $\MaxChain(\GeneralPoset)$  and have a fixed coordinate sum $S$.
  The matrix $A$ represents  a monomial map
  from the polynomial ring $\KK[\,p\,]$ in the unknowns $p_\pi$, $\pi \in \MaxChain(\GeneralPoset)$,
  to a suitably chosen second polynomial ring.
    
  Any data set gives a function $u :    \MaxChain(\GeneralPoset) \mapsto \NN$,
  where $u(\pi)$ is the number of times the permutation $\pi$ has been observed.
  Thinking of $u$ as a column vector, we can form the matrix-vector product $A u$,
  whose entries are the sufficient statistics of the data set $u$.
  The coordinate sum of the vector $Au$ is equal to $S$ times
  the sample size $\,N = \sum_{\pi \in \MaxChain(\GeneralPoset)}  u(\pi)$.

  \begin{itemize}
    \item[(a)] In the {\em \Bottom model}, the sufficient statistic $Au$ 
	  records, for any given poset element $a \in \GeneralPoset$, the number of
	  observed maximal chains 	$\pi$ that pass though $a$.
	  The model parameters are represented by unknowns $c_{a}$,
	  and the monomial map is
	  $$ \qquad p_\pi \,\mapsto \,c_{a_0} c_{a_1} \cdots c_{a_n} \qquad \hbox{for} \,\,\,
	  \pi = (a_0 {<} a_1 {<} \cdots {<} a_n). $$
    \item[(b)] In the {\em \Csiszar model,} the sufficient statistic $Au$ 
	  records, for any cover $a < b$, the number of
	  observed maximal chains 	$\pi$ passing though $a$ and $ b$.
	  The model parameters are represented by unknowns $\,d_{a<b}\,$
	  for $(a{<}b) \in \Cov(\GeneralPoset)$. The  monomial map~is
	  $$ \qquad p_\pi \,\mapsto\, d_{a_0 < a_1} d_{a_1 < a_2} \cdots d_{a_{n-1} < a_n}
	  \qquad \hbox{for} \,\,\,
	  \pi = (a_0 {<} a_1 {<} \cdots {<} a_n). $$
  \end{itemize}

 If $\GeneralPoset = 2^{[n]}$, the Boolean lattice of 
  subsets $a \subseteq [n]$,
  then the  maximal chains $\pi$ in $Q$ are
  identified with permutations in $\symm_n$, and 
    we recover the \Bottom model as defined in Section 2.
Likewise we recover the  \Csiszar model on $\symm_n$ by
setting $d_{a<b}=
  d_{a|i}$ for $\{i\} = b \backslash a$.

  The \Birkhoff and \Inversion model cannot be formulated in the above generality.
  For these we need assume that the poset  $\GeneralPoset$ is a {\em distributive lattice}.
  This means that   $\GeneralPoset = \Order(\PartialRanking)$ 
  is the poset of order ideals in a given partial order $\PartialRanking$ on 
  $[n]$. We refer to $\PartialRanking$ as the {\em constraint poset}. 
  The constraint $i {<} j$ stipulates  that item $i$ must always
  be ranked before item $j$. The maximal chains $\pi$ in
  $\GeneralPoset = \Order(\PartialRanking)$ are the
  permutations of $[n]$ that respect all constraints  in $\PartialRanking$.
  See   \cite{BeerenwinkelErikssonSturmfels2006}
  for an introduction to distributive lattices in a context  of statistical interest.
    
The  compatible permutations $\pi$ are known as {\em linear extensions} of $\PartialRanking$.
    From now on we abbreviate $\LinExt(\PartialRanking) = \MaxChain(\Order(\PartialRanking))$,
    and we identify  elements of $\LinExt(\PartialRanking)$
   with permutations $\pi \in \symm_n$ that represent
  linear extensions of $\PartialRanking$. 
   This allows us to define our other two toric models:

  \begin{itemize}
    \item[(c)] 
      In the {\em \Birkhoff model},
	  the sufficient statistic $Au$ 
	  records, for all $i,j \in [n]$, the number
	  of samples $\pi \in \LinExt(\PartialRanking)$ for which \object $j$ is ranked 
	  in position $i$.  The model parameters are represented by unknowns
	  $a_{ij}$ for  $i,j \in [n]$. The monomial map~is
          $$ \qquad p_\pi \,\mapsto\, a_{1\pi(1)} a_{2\pi(1)} \cdots a_{n \pi(n)} 
             \qquad \hbox{for} \,\,\,\pi \in \LinExt(\PartialRanking). $$
    \item[(d)] 
	  In the {\em \Inversion model},
	  the sufficient statistics $Au$ 
      records, for each ordered pair $i,j $ in $[n]$,
      the number of samples $\pi \in \LinExt(\PartialRanking)$
      for which $i<j$ but      $\pi^{-1}(i) > \pi^{-1}(j)$.
      The model parameters are represented by unknowns
	  $u_{ij}$ and $v_{ij}$.	  The monomial map~is
	  $$ \qquad p_\pi \,\ \mapsto \prod_{{1 \leq i < j \leq n} \atop {\pi^{-1}(i) < \pi^{-1}(j)}} 
	  \!\! u_{ij}	  \prod_{{1 \leq i < j \leq n} \atop {\pi^{-1}(i) > \pi^{-1}(j)}} \!\! v_{ij}
\qquad \hbox{for} \,\,\,\pi \in \LinExt(\PartialRanking). $$
     \end{itemize}

  In general, we have the following inclusions among the four toric models (a)-(d).
  These inclusions of toric varieties correspond to linear projections among the
  model polytopes.

  \begin{Theorem} \label{thm:inclusions}
      \begin{itemize}
	  \item[(i)] The   \Bottom model and 
        the  \Csiszar model 
        on a poset $\GeneralPoset$ satisfy
        $$\Model_{\bott} \subseteq \Model_{\csi},$$
        provided $Q$ has either a unique minimal element $\hat 0$
        or a unique maximal element $\hat 1$.
   	  \item[(ii)] If $Q = \mathcal{O}(\PartialRanking)$ is a distributive lattice,
	    then the \Birkhoff model $\Model_{\birk}$,
        the \Inversion model  $\Model_{\inv}$,
        the \Bottom model $\Model_{\bott}$ and
        the  \Csiszar model $\Model_{\csi}\,$ satisfy
        $$   \Model_{\inv} \subseteq \Model_{\csi}
             \,\, \,\,\text{and} \,\,\,\,
             \Model_{\birk} \subseteq 
             \Model_{\bott} \subseteq 
             \Model_{\csi}.        $$
      \item[(iii)] The inclusions (ii) are strict in general. Moreover, if
      $n \geq 4$ and	    $\GeneralPoset = 2^{[n]}$ then
        $$           \Model_{\inv} \not\subset \Model_{\bott}
\,           \,\,\, \text{and} \,\,\, \,          \Model_{\birk} \not\subset \Model_{\inv}.        $$
    \end{itemize}
  \end{Theorem}

  \begin{proof}
    We begin by establishing (iii). The fact that the inclusions in (ii)
    are strict follows from \ref{ex:strict}. For the second part
    of (iii) consider $n=4$.
    A direct computation as in Section 6 reveals that
    the inversion model $\Model_{\inv}$ is a projective toric variety
    of dimension $6$ and degree $180$ in $\PP^{23}$.
    The Markov basis of $I_{\inv}$ consists of $81$ quadrics.
    Since $\Model_{\birk}$ has dimension $9$, we conclude that
    $\,  \Model_{\birk} \not\subset \Model_{\inv}$.
    An explicit point $p$ in $ \Model_{\birk} \backslash \Model_{\inv}$
    is the uniform distribution on the  nine derangements. This arises
   by setting $a_{ii} = 0$ for all $i$ and $a_{ij} = 1/\sqrt{3}$ for 
all $i \not= j$. The
    quadric $\,  p_{1243} p_{4321} - p_{2143} p_{4312}  \in I_{\inv}\,$
    does not vanish for this particular distribution. 
    
    The ascending model $\Model_{\bott}$ 
    has dimension $11$ and degree $808$. 
    The Markov basis of its toric ideal $I_\bott$ consists of
    six quadrics, $64$ cubics and $93$ quartics.
    One of the cubics is
    \begin{equation}
      \label{eq:botcubic}
      p_{1234} p_{1342} p_{1423} \,-\, p_{1243} p_{1324} p_{1432} \,\, \,\, \in \,\, I_{\bott}. 
    \end{equation}
    An example of a point in
    $ \Model_{\inv} \backslash \Model_{\bott}$ is obtained
    by taking the parameter values
    $$      u_{12} = u_{13} = u_{14} = 0 ,\,
      u_{23}  = u_{24} =  u_{34} =  v_{12} = v_{13} =  v_{23} = v_{24} = 1,\,
      v_{34} = 2,\,      v_{14} = 1/9.    $$
    The resulting distribution is supported on the six permutations in
    \ref{eq:botcubic}. Its coordinates are
    $$p_{1234} = p_{1342} = p_{1423} = 2/9 \quad \text{and} \quad
      p_{1243} = p_{1324} = p_{1432}  =  1/9. $$
    This distribution is not a zero of \ref{eq:botcubic}, and hence it is
    not in the \Bottom model $ \Model_{\bott}$.
  
    The two probability distributions on permutations seen
    above can be lifted to similar counterexamples for $n \geq 5$,
    and we conclude that the non-inclusions are valid for all $n \geq 4$.
  
    The inclusion       $\Model_{\bott} \subset \Model_{\csi}$ in (i) is seen
    by the specialization of parameters 
    that sends $d_{a < b}$ to $c_a$ if $\GeneralPoset$ has a unique maximal element $\hat 1$
    and to $c_b$ if $\GeneralPoset$ has a unique minimum~$\hat 0$.

We lastly prove   the inclusions in (ii).
	The parameters for the  \Csiszar model 
    $\Model_{\rm csi} $ are $d_{a <b}$ where
    $a < b \in \Cov(\GeneralPoset)$.
	If $\MaxChain(\GeneralPoset) = \LinExt(\PartialRanking)$ 
	then	the cover relation $a < b$ means $b = a \cup \{j\}$.
    Thus the following specialization 
    of parameters gives the
    parameterization of $\Model_{\inv}$:
    $$ d_{a < b}\, \,\, \mapsto \,
      \prod_{i \in a , i < j} \!\! u_{ij} 
      \prod_{i \in a , i > j} \!\! v_{ij} . 
    $$
    This shows that the \Inversion model 
    $  \Model_{\rm inv}$ is a subvariety of the
    \Csiszar model  $\Model_{\csi} $.
   
    It remains to show that
    $\,\Model_{\birk} \subset   \Model_{\bott} $. 
    To do this, we let $A$ denote the model matrix for $\,\Model_{\birk}\,$
    and $B$ the model matrix for $\,\Model_{\bott}$. Both matrices
    have their entries in $\{0,1\}$ and they have  $|\LinExt(\PartialRanking)|$ columns.
    The rows $A_{ij}$ of $A$ are indexed by unordered pairs $i,j \in [n] \times [n]$,
    and the rows $B_I$ of $B$ are indexed by subsets of $[n]$. We have the identity
    $$ 
      \begin{matrix}
         A_{ij} \,\, = \,\,\, \sum \bigl\{B_I\,:\,I \in \binom{[n]}{j} \,\, {\rm and} \,\, i \in I \bigr\}
         \,-\,\,  \sum \bigl\{B_I\,:\,I \in \binom{[n]}{j-1} \,\, {\rm and} \,\, i \in I \bigr\}
      \end{matrix}.
    $$
    This shows that every row of $A$ is a $\ZZ$-linear combination of the rows of $B$.
    Hence, the kernel of $A$ contains the kernel of $B$, and this implies that
    the toric ideal $I_A = I_{\birk}$ contains the toric ideal $I_B = I_{\bott}$.
    We conclude that  $\Model_{\birk}$ is a submodel of $\Model_{\bott}$.
  \end{proof}

  In the rest of this paper we consider the  \Bottom and \Csiszar models
  only in the graded situation, that is, when the monomial images
  of all the unknowns   $p_c$, $c \in \MaxChain(\GeneralPoset)$,
  have the same total degree. The latter is equivalent to requiring
  that all maximal chains in $\GeneralPoset$
  have the same cardinality, which in turn is equivalent to $\GeneralPoset$
  being graded. For a graded poset $\GeneralPoset$ we denote by
  $\rk : \GeneralPoset \rightarrow \NN$ its rank function
  and write $\GeneralPoset_i$ for the
  set of its elements of rank $i$. By $\rk(\GeneralPoset)$ we denote the rank of 
  $\GeneralPoset$, which is the maximal rank of any of its elements.

  In the next three sections we undertake a detailed study of the models (b), (a) and (d), in this order.
  The Birkhoff model (c) has already received considerable attention in the
  literature \cite{DiaconisEriksson2006, DiaconisSturmfels1998}, 
  at least for  $\LinExt(\PartialRanking) = \symm_n$, and we content ourselves
  with a few brief remarks. Its model polytope,    the Birkhoff polytope
  of doubly stochastic matrices,
  is a key player in  combinatorial optimization,
  and it is linked to many 
  fields of pure mathematics.
  
  The restriction of the Birkhoff model and its polytope to proper subsets $\LinExt(\PartialRanking)$ 
  of $\symm_n$ has been studied only in some special cases.
  For example, Chan, Robbins and Yuen \cite{ChanRobbinsYuen1999} considered this
  polytope for the constraint poset $\PartialRanking$ given by
  the transitive closure of $j > j-2$ and $j > j-3$ for  $3 \leq j \leq n$.
  They stated a conjecture  on its volume which was proved by Zeilberger
  \cite{Zeilberger1999}.
  We close by noting a formula for the dimension of these polytopes.

  \begin{Proposition}
    \label{prop:birkhoffdim}
    Let $\PartialRanking$ be an arbitrary constraint poset on $[n] = \{1,2,\ldots,n\}$.    Set 
    $$ Z \, =\, \left\{ \,(i,j) \in [n] \times [n]~|~ \pi(i) \neq j \mbox{ for all~} 
                                         \pi \in \LinExt(\PartialRanking)\, \right\} \qquad $$
    $$ \text{and} \qquad C\, = \,\left\{ (i,j) \in [n] \times [n]~|~ (i,j) \not\in Z \mbox{~ and~} 
    {{(i,j') \in Z \mbox{~for~some~} j' > j \mbox{~or} } \atop  
     {(i',j) \in Z \mbox{~for~some~} i' > i \mbox{\phantom{~or}}}}  \right\}.  $$
    The model polytope $\BirkhoffPol$ of the Birkhoff model, expressed using
     coordinates $x_{ij}$ on $\RR^{n \times n}$, equals the face of 
    the classical Birkhoff polytope of bistochastic $n {\times}n $-matrices defined by
        \begin{equation}
       \label{eq:birkhoff}        x_{ij} = 0 \quad \text{    for all $\,(i,j) \in Z$.} 
\end{equation}
    In particular, the dimension of the Birkhoff model polytope is    
     $\,{\rm dim}(\BirkhoffPol) \,=\,n^2 - |Z| - |C|$.
  \end{Proposition}
  \begin{proof}
    Clearly, the model polytope $\BirkhoffPol$ of the \Birkhoff model is contained in the
    classical    \Birkhoff polytope. Equally obvious is that all equations \ref{eq:birkhoff} are valid 
    for the model polytope.
    Hence $\BirkhoffPol$ is contained in the polytope cut out from the classical
    Birkhoff polytope by \ref{eq:birkhoff}.

    Following the lines of the  Birkhoff-von Neumann Theorem (see e.g. \cite[(5.2)]{Barvinok2002}),
    we note that the vertices of the polytope cut out by \ref{eq:birkhoff} from the classical Birkhoff polytope
    are the permutation matrices for the permutations $\pi \in \LinExt(\PartialRanking)$. 
    The first assertion now follows.
    
    The linear relations on the Birkhoff polytope
    state that all row and column sums are~$1$. \
  We set      $x_{ij} = 0$ for $(i,j) \in Z$.
  In the resulting linear relations precisely
  the variables $x_{ij}$ for $(i,j) \in C$ are the leading terms with respect to 
  order of the variables induced by the lexicographic order on the index tuples.
  This proves the dimension statement.
  \end{proof}
   
   We  illustrate~\ref{prop:birkhoffdim} with two simple examples.
      If $\PartialRanking$ is an $n$-element antichain then $Z= \emptyset$ and 
  $C = \{(1,n),(2,n), \ldots , (n,n), (n,n-1), \ldots (n,1)\}$. Here
  our formula gives    the dimension $\,n^2-0-(2n-1) = (n-1)^2$ of the classical
  Birkhoff polytope. If $\PartialRanking$ is the $n$-chain $1 {<} 2 {<} \cdots {<} n$
     then $Z = \{ (i,j) \in [n] \times [n]~ |~ i \neq j\}$ and 
  $C = \{ (i,i) | i \in [n]\}$. Here the model polytope is just one point, since
  $\,{\rm dim}(\BirkhoffPol) =   n^2 - |Z| - |C| = n^2-n(n-1)-n = 0$.    

\section{The \Csiszar model}

  The \Csiszar model for the Boolean lattice $\GeneralPoset = 2^{[n]}$ was
  studied by Vill\~o Csisz\'ar in \cite{Csiszar2009a, Csiszar2009b}.
  She calls it the {\em L-decomposable model} where the letter  ``L'' refers to Luce \cite{Luce1959}.
  Indeed, the model can be seen as the generic model satisfying Luce-decomposability (see \cite{Marden1995}).
  We prefer to call it the {\em \Csiszar model}, to credit her work for
  introducing this model into algebraic statistics.
We note that the \Csiszar model for $\GeneralPoset =2^{[n]}$
also appears in work on multiple testing procedures by Hommel {\it et.~al.}
  \cite{HommelBretzMaurer2007}, but with a different
  coordinatization of its model polytope.
    Throughout this section, we fix a graded poset $\GeneralPoset$ of positive rank.

  We begin by defining a $0$-$1$-matrix $A=\CsiszarPol$ that represents the Csisz\'ar model.
  Our construction is based on the technique employed for
  $Q = 2^{[n]}$ in  Csisz\'ar's proof of \cite[Theorem 1]{Csiszar2009a}.
  The columns of $\CsiszarPol$ are indexed by the unknown probabilities $p_\pi$
  where $\pi \in \MaxChain(\GeneralPoset)$, and the rows of $\CsiszarPol$ are
  indexed by the model parameters  $d_{a < b}$ where
  $(a {<} b) \in \Cov(\GeneralPoset)$. 
  We write $\MinCov(Q)$ for the set of cover relations $a < b$ for some
  element $a \in \GeneralPoset_0$   of rank~$0$.     
  
  Consider the discrete undirected graphical model   \cite{DrtonSturmfelsSullivant2009, 
  GeigerMeekSturmfels2006}  given by the $n$-chain
  graph  $G = ([n],E)$ with edge set $E = \{ \{ i,i+1\}~ |~1 \leq i \leq n-1\}$.
  We take as the states of node $i$ the set
  $\GeneralPoset_i$ of all elements of rank $i$ in $\GeneralPoset$. 
  The $n$-chain graph $G$ is chordal (or decomposable), so the
  five equivalent conditions of \cite[Theorem 4.4]{GeigerMeekSturmfels2006}  hold for $G$.
  Let $A_G$ denote the associated model matrix \cite[\S 2.2]{GeigerMeekSturmfels2006}.
  It has $\prod_{i=0}^n |\GeneralPoset_i|$ columns indexed by tuples $(a_0,\ldots, a_n)$  of
  elements $a_i \in \GeneralPoset_i$ and $\sum_{i=0}^{n-1} |\GeneralPoset_i| \cdot |\GeneralPoset_{i+1}|$ 
  rows indexed by pairs $(a,b)$ of elements of $\GeneralPoset$ from consecutive ranks.
  Its entries are $0$ or $1$ according to the pattern for an undirected graphical model.
  More precisely, in row $(a,b)$ all columns are $0$ except for the rows indexed by tuples
  containing $a$ and $b$. 
  We shall use the following  key facts from
    \cite[Theorem 4.4]{GeigerMeekSturmfels2006}: the image of
  the monomial map given by $A_G$ is closed, and
  the cone spanned by the columns of  
  $A_G$ contains all non-negative vectors in the column space  of $A_G$. 

  As in Csisz\'ar's proof of  \cite[Theorem 1]{Csiszar2009a},
  we focus on the submatrix $A'_G$ of $A_G$
  whose column labels $(a_0,\ldots, a_n)$ correspond
  to  maximal chains $a_0 <  \cdots <  a_n$ from $\MaxChain(\GeneralPoset)$. 
  Many of the rows of $A'_G$ are entirely
  zero, namely, all those rows indexed by pairs $(a,b)$,
  where $a$ is not covered by $b$ in $\GeneralPoset$. Let $A''_G$ denote the matrix
  obtained from $A'_G$ by deleting all such zero rows. The remaining
  rows are indexed by pairs $(a,b) \in \GeneralPoset_i \times \GeneralPoset_{i+1}$
  for some $i$.
  Equivalently, the rows of $A''_G$ are indexed by $\Cov(\GeneralPoset)$.
  This shows that the toric model $A''_G$ is precisely
  our \Csiszar model, and, with this identification of coordinates, our
  polytope $\CsiszarPol$ coincides with the convex hull of the columns of  
  $A''_G$.  Now we are in a position to give a description of the model polytope
  $\CsiszarPol$  in terms of   linear equalities and inequalities.

  \begin{Theorem} \label{thm:csiszarpolytope}
    Let $\GeneralPoset$ be a graded poset of rank $\geq 1$ and
    $\CsiszarPol \subseteq \RR^{\Cov(\GeneralPoset)}$  the model polytope of its
    Csisz\'ar model, with coordinates 
	$x_{a < b}$ indexed by cover relations $a < b $ in $\Cov(\GeneralPoset)$.
    Then $\CsiszarPol$ is of dimension 
	$|\Cov(\GeneralPoset)| - |\GeneralPoset| + |\GeneralPoset_n| + |\GeneralPoset_0|-1$. 
	Inside the orthant defined by
    \begin{equation}
        \label{csiszar:facets} 
       x_{a < b} \,\geq \, 0 \qquad \hbox{for} \,\,\,\, (a < b) \in \Cov(\GeneralPoset),
    \end{equation}
    the polytope $\CsiszarPol$ is the solution set of the inhomogeneous linear equation
    \begin{eqnarray} \label{csiszar:equalities-b}
        \sum_{a < b \in \MinCov(\GeneralPoset)} \!\!\! x_{a < b} \,\,\,\, =  \,\,\,\, 1
    \end{eqnarray}
    together with the system of linear homogeneous equations
    \begin{eqnarray} \label{csiszar:equalities-a}
      \sum_{b \in \nabla a} x_{a < b} \,\,=\,\, \sum_{b \in \Delta a} x_{b < a} 
      \,\,\quad \hbox{for} \,\,\,\, a \in \GeneralPoset \backslash (\GeneralPoset_0 \cup \GeneralPoset_n),
    \end{eqnarray}
    where $\nabla a$ is the set of $b$ that cover $a$, and
    $\Delta a$ is the set of $b$ that are covered by $a$.
  \end{Theorem}

  \begin{proof}
    Let $G = ([n],E)$ with $E = \{ \{ i,i+1\}~ |~1 \leq i \leq n-1\}$
    be the $n$-chain and $A_G$ the defining matrix of its graphical model as
    discussed above. Also let $A_G'$ and $A_G''$ be as above. 

    The $n$-chain graph $G$ is decomposable, so the five equivalent conditions in
    \cite[Theorem 4.4]{GeigerMeekSturmfels2006} are true.
    The fifth condition, that the exponential family is closed in
    the probability simplex, is equivalent to the
    statement that the model polytope of that $n$-chain model is defined
    by linear equations and non-negativity constraints only.
    See \cite{KatTho2007} for a toric algebra perspective.
    We have shown that the toric model of $A''_G$
    is our \Csiszar model. With this identification, the model
    polytope $\CsiszarPol$ coincides with the convex hull of the columns of  	$A''_G$.

    The matrix $A'_G$ was constructed so that its columns are precisely the
    points on a face of the model polytope for $A_G$. Hence the model polytope of the 
    \Csiszar model is obtained from the earlier polytope by simply setting  
    some of the non-negative coordinates to zero. This implies that $\CsiszarPol$ inherits
    all the desirable properties spelled out in Theorem 4.4 of \cite{GeigerMeekSturmfels2006}.
    In particular, its exponential family is closed,  and the polytope $\CsiszarPol$
    coincides with the set of all non-negative points in the affine space
    spanned by the columns of the matrix $A''_G$.
%         This important property was called {\em face completeness} by
%         Katsabekis and Thoma \cite{KatTho2007}, so what we have argued here
%		 is that all three matrices $A_G$, $A_G'$ and $A_G''$ are face complete.

    At this stage we only need to show that the affine span 
    of the columns of $A''_G$ equals
    the solution space of \ref{csiszar:equalities-b} and \ref{csiszar:equalities-a}.
    The equation \ref{csiszar:equalities-b} holds for a vertex of the model polytope because any maximal chain  
    contains exactly one cover relation involving an element of rank $0$ and an element of rank $1$.
    The equations \ref{csiszar:equalities-a} hold for a vertex of the model polytope because,
    given any element $a \in \GeneralPoset$, a maximal chain
    either contains no cover relation involving $a$ or exactly two, one of the form $b < a$ and one
    of the form $a < b'$. Hence each column of $A''_G$ satisfies 
    \ref{csiszar:equalities-b} and \ref{csiszar:equalities-a}. Conversely, any $0$-$1$-solution
    of these equations must come from a maximal chain in $\GeneralPoset$, and hence
    is among the columns of $A''_G$.
  \end{proof}

\begin{Remark} The maximal likelihood estimator $\hat {\bf p}$ for the
 \Csiszar model is a rational function of the sufficient statistics ${\bf b}$.
 Indeed, as for any toric model \cite[Theorem 1.10]{PachterSturmfels2005},
 the MLE   is the unique non-negative real solution of the linear equations
 $A''_G \cdot {\bf p} = {\bf b}$ where  ${\bf p} \in V(I_{\csi})$.
 An explicit formula for $\hat {\bf p}$ as a rational function in ${\bf b}$ is obtained
 from the corresponding formula for the $n$-chain model $A_G$ by setting the
redundant sufficient statistics to zero.  This specialization works because
 the \Csiszar model is a face of the $n$-chain model, and all formulas are compatible
 with our transition from $A_G$ to $A''_G$ via $A'_G$.
 On the other hand, the same idea of computing the MLE rationally
 by restriction no longer works for our submodels of the \Csiszar model, such as the 
 Birkhoff model or the ascending model. For instance, for $n=3$,  the
 matrix  $A''_G $ is invertible and $ \hat {\bf p} = (A''_G)^{-1}  {\bf b}$,
 while the MLE for $ I_{\rm \birk}   = I_{\bott} $ requires Cardono's formula:
 we must solve a cubic equation to get the MLE. \qed
\end{Remark}
  
  The toric ideal $I_{\csi}$ of the \Csiszar model is the kernel of the ring homomorphism
  $$ \qquad \KK[p] \rightarrow \KK[ d] \,,\,\,\,
  p_\pi \,\mapsto\, d_{a_0 < a_1} d_{a_1 < a_2} \,\cdots\, d_{a_{n-1} < a_n}
  \qquad \hbox{for} \,\,\,
  \pi = (a_0 {<} a_1 {<} \cdots {<} a_n). $$
  The minimal generators of $I_{\csi}$ 
  form the Markov basis of $\Model_{\csi}$.
  As shown in the proof of \ref{thm:csiszarpolytope}, the \Csiszar model polytope 
  $\CsiszarPol = A''_G$  inherits the  equivalent
  conditions (b),(c),(d),(e) in \cite[Theorem 4.4]{GeigerMeekSturmfels2006}
  from the larger model $A_G$.  In particular, the toric ideal $I_{\csi}$ has a Gr\"obner basis 
  consisting of quadratic binomials.  
  We shall now describe this Gr\"obner basis explicitly.
  It generalizes the Markov basis  for  
  $\GeneralPoset = 2^{[n]}$ in \cite[Theorem 3.1]{Csiszar2009a}.

  \begin{Theorem}
    \label{thm:csiszargroebner}  
	A Gr\"obner basis for the toric ideal $I_{\csi}$ of the \Csiszar model
    on a graded poset $\GeneralPoset$ is given by all quadratic binomials of the form
    \begin{equation}
      \label{eq:csiszarquadric}
	  p_{\pi_1 \pi_2} \cdot p_{\pi_1' \pi_2'} \,-\, p_{\pi_1\pi_2'} \cdot p_{\pi_1' \pi_2}, 
    \end{equation}
	where the chains $ \pi_1$ and $\pi_1'$ have the same ending point and
	both $\pi_2$ and $\pi_2'$ start there.
  \end{Theorem} 

  \begin{proof}
    It is easy to check that the binomial quadrics that lie in the ideal $I_{\csi}$
    are precisely the quadrics \ref{eq:csiszarquadric}. These are inherited
    from the conditional independence statements valid for the $n$-chain graphical model $G$.
    These statements translate into a quadratic 
    Gr\"obner basis for the toric ideal of the matrix $A_G$. The leading terms
    of that Gr\"obner basis are squarefree, so by \cite[Corollary 8.9]{Sturmfels1996} they define a regular
    unimodular triangulation of the convex hull of the columns of $A_G$.
    Since $\CsiszarPol = A''_G$  is a face of that polytope, that face inherits
    the regular unimodular triangulation from $A_G$. We conclude that the
    Gr\"obner basis which specifies this regular triangulation of $\CsiszarPol$
    consists precisely of the quadrics  \ref{eq:csiszarquadric}. 
  \end{proof}

  The Gr\"obner basis \ref{eq:csiszarquadric}  reveals that the
    \Csiszar model has desirable algebraic properties:
 
  \begin{Corollary}
    The coordinate ring $\KK[\,p\,] / I_{\csi}$ of the \Csiszar model over any field $\KK$ is
	Cohen-Macaulay and Koszul. Its Krull dimension  equals
	$|\Cov(\GeneralPoset)| - |\GeneralPoset| + |\GeneralPoset_n| + |\GeneralPoset_0|$.
  \end{Corollary}

  \begin{proof}
    Since $I_{\csi}$ has a quadratic Gr\"obner basis, by \ref{thm:csiszargroebner},
    it follows that $\KK[\,p\,] / I_{\csi}$
    is Koszul. Again by \ref{thm:csiszargroebner} there is a squarefree initial ideal 
    of $I_{\csi}$. Hence by \cite[Proposition 13.15]{Sturmfels1996} the polytope 
    the semigroup algebra $\KK[\,p\,] / I_{\csi}$ is normal. 
    and hence Cohen-Macaulay, by Hochster's Theorem \cite[Theorem 1]{Hochster1972}.
    The dimension of this semigroup algebra is one more than the
    dimension of its polytope, given in \ref{thm:csiszarpolytope}.
  \end{proof}
    
  For computations it is convenient to represent
  the quadrics in \ref{eq:csiszarquadric} as the $2 {\times} 2$-minors of 
  certain natural  matrices $M_q$
  that are indexed by the elements $q$ of the poset $\GeneralPoset$.
  The row labels of the matrix $M_q$ are the maximal chains in the order ideal
  $\GeneralPoset_{\leq q} = \{a \in \GeneralPoset : a \leq q \}$
  and the column labels of $M_q$ are the maximal chains in the  filter
  $\GeneralPoset_{\geq q} = \{b \in \GeneralPoset  :q \leq b \}$.
  Thus $M_q$ is a matrix of format $|\MaxChain(\GeneralPoset_{\leq q})| \times 
  |\MaxChain(\GeneralPoset_{\geq q}) |$.
  We define $M_q$ as follows. The entry of $M_q$ in the row labeled
  $\pi_1 \in \MaxChain(\GeneralPoset_{\leq q})$
  and the column labeled $\pi_2 \in \GeneralPoset_{\geq q}$  is the unknown $p_\pi$
  where $\pi$ denotes the maximal chain of $\GeneralPoset$ that is
  obtained by concatenating $\pi_1$ and $\pi_2$.

  \begin{Corollary}
    The Markov basis of the Csisz\'ar ideal $I_\csi$ consists of the
    $2 {\times} 2$-minors of the matrices $M_q$, where $q$ runs over $\GeneralPoset$.
    This Markov basis is also a Gr\"obner basis.
  \end{Corollary}

  \begin{proof}
    Each $2 \times 2$-minor of $M_q$ has the form required in \ref{eq:csiszarquadric},
    and, conversely, each binomial in \ref{eq:csiszarquadric}
    occurs  as a $2 {\times} 2$-minor of $M_q$ for some $q$. Note that this
    element $q \in Q$ is generally not unique for a given binomial.
     The Gr\"obner basis statement is a part of  \ref{thm:csiszargroebner}.
  \end{proof}
    
  We illustrate our results for the case when
  $\GeneralPoset = 2^ {[n]}$ is the Boolean lattice, with $n \leq 6$.
  For $n = 3$, the ideal $I_\csi$ is zero as seen in Section 2.
  For $n = 4$, the ideal     $I_{\csi}$ is the complete intersection
  of six quadrics, namely, the  determinants of the six $2 {\times} 2$-matrices  $M_{\{i,j\}}$.
  Geometrically, these    correspond to the
  six square faces of the $3$-dimensional permutahedron:
  $$
    \begin{matrix} I_{\rm csi} \, = \,\,\langle \!\! &
        p_{1243} p_{2134} -p_{1234} p_{2143} , &
        p_{1342} p_{3124} -p_{1324} p_{3142}, &
        p_{1432} p_{4123} -p_{1423} p_{4132}, & \\ & 
        p_{2341} p_{3214}-p_{2314} p_{3241},&
        p_{2431} p_{4213}-p_{2413} p_{4231},&
        p_{3421} p_{4312}-p_{3412} p_{4321} & \!\! \rangle .
    \end{matrix}
  $$
  We conclude that the \Csiszar model for $n=4$ has dimension $17$, as predicted by \ref{thm:csiszarpolytope}.   As a projective variety, this model has degree $32$ since it is a complete intersection.
%  its Hilbert series equals $$ \frac{(1+t)^6}{(1-t)^{18}}. $$
    For $n = 5$, the Markov basis consists of the $2 {\times} 2$-minors
  of the ten $2 {\times} 6$-matrices $M_{\{1,2\}}, M_{\{1,3\}}, \ldots, M_{\{4,5\}}$
  and ten $6 {\times} 2$-matrices $M_{\{1,2,3\}}, M_{\{1,2,4\}}, \ldots, M_{\{3,4,5\}}$.
  For example,
  $$
    M_{\{2,4\}} \quad = \quad
    \begin{pmatrix}
      p_{24135} & p_{24153} & p_{24315} &  p_{24351} & p_{24513} & p_{24531} \\
      p_{42135} & p_{42153} & p_{42315}  & p_{42351} & p_{42513} & p_{42531}
    \end{pmatrix}.
  $$
  Altogether, these matrices  have $300$ maximal minors
  but $30$ of the minors occur in two matrices, so the total number
  of distinct Markov basis elements is $270$.
  The dimension of this model is $49$, and its degree equals $50493797160$.
  The Hilbert series of $\KK[\,p\,]/I_\csi$ equals
  \begin{small}
     $$ 
       \begin{matrix}
         (1+70 t+2215 t^2+42020 t^3+534635 t^4+4837694 t^5+32227985 t^6+161529320 t^7 \\ 
         +617560160 t^8 + 1816401720 t^9+4129171068 t^{10}+7265606880 t^{11}+9880962560 t^{12} \\ 
         +10337876480 t^{13}+8250364160 t^{14}+4953798656 t^{15}+2189864960 t^{16} \\ +688455680 t^{17} +
         145162240 t^{18}+18350080 t^{19}+1048576 t^{20})/(1 - t)^{50}.
       \end{matrix}
     $$
  \end{small}

  For $n=6$, the Markov basis is represented by
  the fifteen $2 {\times} 24$-matrices $M_{\{i,j\}}$,
  the twenty $6 {\times} 6$-matrices $M_{\{i,j,k\}}$
  and the fifteen $24 {\times} 2$-matrices $M_{\{i,j,k,l\}}$.
  Altogether, these $50$ matrices have $12780$ minors of size $2 {\times} 2$
  but only $10980$ of the binomial quadrics are distinct.

  A systematic way of understanding our matrices $M_q$
  is furnished by Sullivant's theory of {\em toric fiber products}
  \cite{Sullivant2007}. This method will become crucial when studying the
  \Bottom model in the next section and we will explain at the end of the section
  how toric fiber product can also be used to give an alternative proof of \ref{thm:csiszargroebner}.

  % For our presentation above,
  % we considered it to be more direct (and statistically more intuitive) to appeal
  % instead to the graphical model of the $n$-chain $G$.
  % In the next section, however, toric fiber products
  % will become crucial in our study of the ascending model on~$\GeneralPoset$.

\section{The \Bottom model}

  At the end of \cite[p.~233]{Csiszar2009a} it is asserted that a Markov 
  basis for the \Bottom model
  on $\GeneralPoset = 2^{[n]}$ can be obtained in a similar way as was done  
  for the standard \Csiszar model, but no details are given.
  However, simple examples show that it does not suffice to consider quadratic 
  binomials for the generating set and it is not clear from  \cite{Csiszar2009a}
  which properties the defining ideals of the \Bottom and \Csiszar model have 
  in common.
  The defining ideal and the model polytope of the \Bottom model seem to be 
  complicated and 
  more interesting than those of the \Csiszar model. These are the structures 
  to be explored in this section.

  Generalizing the notation introduced in the preceding section, for any subset 
  $A \subseteq \GeneralPoset$, we consider the  set of elements of $A$ that 
  cover an element from~$A$:
  $$\nabla A \,:=\, 
    \{ b \in \GeneralPoset~|~a < b \in \Cov(\GeneralPoset) 
                               \mbox{~for some~} a \in A \}.$$
  We also consider the set of elements covered by an element from $A$:
  $$\Delta A :=\,  
    \{ b \in \GeneralPoset~|~b < a \in \Cov(\GeneralPoset) 
                               \mbox{~for some~} a \in A \}.$$ 
 
%    If $A = \{ a \}$ is a singleton then 
%  we write $\nabla a$ for $\nabla A$ and $\Delta a$ for $\Delta A$. 

  \begin{Theorem} \label{thm:bottompolytope}
    Fix a graded poset $\GeneralPoset$ of rank $n$. The model polytope
    $\BottomPol $ of the ascending model is the set of solutions in the space
      $   \RR^{|\GeneralPoset|}$, with coordinates $x_a$ for
    $a \in \GeneralPoset$, of the~equations
        \begin{eqnarray} \label{bottom:equalities}
      \sum_{a \in \GeneralPoset_i} x_a & =  \,\, 1, &\, 0 \leq i \leq n,
    \end{eqnarray}
    and the inequalities 
    \begin{eqnarray} \label{bottom:facets} 
      x_a & \geq \,\, 0, & \, a \in \GeneralPoset, \label{bottom:ineq1} \\
      -\sum_{a \in A} x_a + \sum_{a \in \nabla A} x_{a} & \geq \,\,
      0, & \, A \subseteq \GeneralPoset_i \,,\, \, 0 \leq i \leq n-1. \label{bottom:ineq2} 
    \end{eqnarray} 
  \end{Theorem}
  
    \begin{proof}
    Equations \ref{bottom:equalities} are valid on every vertex of $\BottomPol$ because
    every maximal chain in $P$ has exactly one element of rank $i$ for all $0 \leq i \leq n$.
		The inequalities \ref{bottom:ineq2} express the fact that
		 if a maximal chain passes through an element of $A 
	\subseteq \GeneralPoset_i$ then it must also pass through a unique
	element of $\nabla A$.
	Inequalities \ref{bottom:ineq1} are obviously valid for 
	$\BottomPol$. 
	Hence $\BottomPol$ is contained in the intersection of the linear
	spaces defined by \ref{bottom:equalities} and the halfspaces defined
	by \ref{bottom:ineq1}~and~\ref{bottom:ineq2}.

    For the converse we proceed by induction on $n$.
    If $n = 0$ then $\BottomPol$ is a simplex of dimension $|Q|-1$,
    defined by  \ref{bottom:equalities} and \ref{bottom:ineq1}.
    If $n=1$ then the result is identical to~\cite[Corollary~1.8~(b)]{OhsugiHibi1998}. 

    Assume $n \geq 2$. Let $\xx = (x_a)_{a \in \GeneralPoset} \in 
    \RR^{\GeneralPoset}$ be any vector satisfying 
    \ref{bottom:equalities}, \ref{bottom:ineq1} and 
    \ref{bottom:ineq2}. Let $\xx'$ be the projection  of $\xx$ onto the 
    coordinates in  	$\GeneralPoset' = 
    \GeneralPoset_0 \cup \cdots \cup \GeneralPoset_{n-1}$ 
    and $\xx''$ the projection of $\xx$ onto
    $\GeneralPoset'' = \GeneralPoset_{n-1} \cup \GeneralPoset_n$.
    By induction, $\xx'$ and $\xx''$ lie in the model polytopes of
    the \Bottom model for $\GeneralPoset'$ and $\GeneralPoset''$. 
    Hence we can write $\xx$ and $\xx'$ as convex linear combinations:
    $$ 
      \begin{matrix}
       & \xx' &=& \sum_{c' \in \MaxChain(\GeneralPoset')} \lambda_{c'} c' &\quad \hbox{and} \quad &
       \xx'' & = &  \sum_{c'' \in \MaxChain(\GeneralPoset'')} \lambda_{c''} c''.
      \end{matrix}
    $$
    Here we identify $c'$ and $c''$ with the $0/1$-vector that has support
    $c'$ and $c''$ respectively.

    Consider a fixed element $a \in \GeneralPoset_{n-1}$. Let
    $c_1', \ldots, c_r'$ be the chains from the above expansion of
    $\xx'$ that contain $A$ and for which $\lambda_{c'} > 0$. Let 
    $c_1'', \ldots, c_s''$ be the chains from the above expansion of
    $\xx''$ that contain $a$ and for which $\lambda_{c''} > 0$.
    The coordinate $x_a'$ of $\xx'$ then equals
    $\sum \lambda_{c_i'}$ and the coordinate $x_a''$ of $\xx''$ equals
    $\sum \lambda_{c_i''}$. Since $x_a'$ and $x_a''$ coincide with
    the coordinate $x_a$ of $\xx$, we have  $\sum \lambda_{c_i'} = \sum \lambda_{c_i''}$. 
    After relabeling
    (and possibly swapping $\xx'$ and $\xx''$) we may assume that $\lambda_{c_1'}$ is
    the minimum of $\{\lambda_{c_1'} , \ldots,
    \lambda_{c_r'}, \lambda_{c_1''}, \ldots, \lambda_{c_s''}\}$.
    Then we replace $\lambda_{c_1''}$ by
    $\lambda_{c_1''}- \lambda_{c_1'}$. Let 
    $c_1 \in \MaxChain(\GeneralPoset)$ be the concatenation of
    $c_1'$ and $c_1''$. Now set $\lambda_{c_1} = \lambda_{c_1'}$ and
    proceed with the new coefficients and the chains
    $c_2', \ldots, c_r'$ and $c_1'', \ldots, c_s''$. 
    Clearly the sums of the coefficients of $c_2', \ldots, c_r'$ 
    and $c_1'', \ldots, c_s''$ still coincide. Proceeding by induction 
    and  summing over all $a \in \GeneralPoset_{n-1}$ for which $\xx_a > 0$,
    one constructs an expansion $\sum \lambda_ic_i$ in terms of
    chains in $\MaxChain(\GeneralPoset)$ whose projection onto
    $M(\GeneralPoset')$ equals $\xx'$ and whose projection onto
    $M(\GeneralPoset'')$ equals $\xx''$. 
    Hence $\xx = \sum \lambda_ic_i$, and we have
    $\lambda_i \geq 0$ and $\sum \lambda_i =
    \sum_{ a \in \GeneralPoset_{n-1}} x_a = 1$ by \ref{bottom:equalities}.
    This proves that $\xx \in \BottomPol$. 
  \end{proof}

  In the preceding proof, when showing that 
  any $\xx$ satisfying \ref{bottom:equalities}--\ref{bottom:ineq2} lies in $\BottomPol$,
    we use \ref{bottom:ineq2} only in the induction   base $n=1$. 
  The equations \ref{bottom:equalities} are complete and independent
  when $Q = 2^{[n]}$ is the Boolean lattice, so in that case the dimension
  of the model polytope $\BottomPol$ is equal to $2^n - n - 1$.
  In general the dimension is more subtle to calculate and we do not know
  any good description. For example if the induced subposet of 
  $\GeneralPoset$ on the elements of two consecutive ranks $i$ and $i+1$
  is disconnected then  $\BottomPol$ is contained in each hyperplane
  defined by the equality of the sum over the variables of rank $i$ and
  $i+1$ in a component.

  Now we turn to the toric ideal $I_{\bott}$   of the \Bottom model. It is the kernel of the map
  \begin{equation}
  \label{eq:tmap}
  \KK[ \,p\,] \rightarrow \KK[\,t \,] ,\,\, p_\pi \,\, \mapsto \,\,
        t_{a_0} t_{a_1} \cdots t_{a_n} \quad
        \text{for} \,\,\pi = (a_0 < \cdots < a_n) \in \MaxChain(\GeneralPoset) . 
        \end{equation}
  If  $\rk(\GeneralPoset) = 0$ then this map is injective and
  $I_{\bott} = \{0\}$, so we assume    $\rk(\GeneralPoset) \geq 1$ from now on.
  The case  $\rk(\GeneralPoset) =1$ serves as the base case 
  for our inductive constructions.   Here the poset $Q$ is identified with 
  a bipartite graph on $Q_0$ and $Q_1$, and the monomial map
  $p_\pi \mapsto t_{a_0} t_{a_1}$  defines the  toric ring 
  associated with a bipartite graph in commutative algebra.
  A generating set of the kernel of this map was determined 
  in \cite[Lemma 1.1]{OhsugiHibi1999} and shown to be a universal 
  Gr\"obner basis in \cite[Proposition 8.1.10]{Villerreal2001}. 
  This result has already proven to be useful in algebraic statistics 
(see e.g.~\cite{FienbergPetrovicRinaldo2010}). 
  
  \begin{Lemma}[Ohsugi-Hibi \cite{OhsugiHibi1999}, Villerreal \cite{Villerreal2001}]
      \label{lem:bottomrank1}
    Let $\GeneralPoset$ be a graded poset of rank $1$. Then a universal Gr\"obner basis
    of the toric ideal $I_{\bott}$
      consists of all cycles in $Q$, expressed as binomials 
	$$p_{a_0 < a_1} p_{a_2 < a_3} \cdots p_{a_{2s-2} < a_{2s-1}}
	 \,-\, p_{a_2 < a_1} p_{a_4 < a_3} \cdots p_{a_{2s} < a_{2s-1}}$$
	where $a_{2s} = a_0$ and the $a_i$ are pairwise distinct otherwise. 
  \end{Lemma}

  Lemma 1.1 in \cite{OhsugiHibi1999} and Proposition 8.1.10 in \cite{Villerreal2001} 
  is actually formulated in a slightly different language. 
  For a graph $G = (V,E)$ with vertex set $V$ and edge set $E$ one considers
  two polynomial rings, one where the variables are indexed by the edges of the graph 
  and one where the variables are indexed by the vertices. Now the edge variables are mapped
  to the product of the corresponding vertex variables. If the graph is bipartite with
  bipartition $V = V_0 \cup V_1$ then one can consider it as a graded poset of
    rank $1$ where vertices from $V_0$ are covered by their neighbors in $V_1$.
    Of course, the role of $V_0$ and $V_1$ can also be reversed.
  Thus the edge variables represent 
  variables indexed by the maximal chains,
   and the kernel of the map to the corresponding product of vertices 
  coincides with the toric ideal $I_{\bott}$. 

  Now we are in a position to describe a Gr\"obner basis for $I_{\bott}$
  when ${\rm rank}(Q) \geq 1$.

  \begin{Theorem}
    \label{thm:bottomideal}
    A Gr\"obner basis for the toric ideal $I_{\bott}$ of the \Bottom model
    on a graded poset $\GeneralPoset$ of rank $n$ 
    is given by two classes of binomials.
    The first class consists of the quadrics
        \begin{equation}
          \label{eq:bottomquadric}
          p_{\pi_1} \cdot p_{\pi_2} \,-\, p_{\bar{\pi}_1} \cdot p_{\bar{\pi}_2}, 
          \end{equation}
	  where $\pi_1,\bar{\pi}_1, \pi_2, \bar{\pi}_2$ are distinct chains of at least three elements,
          such that $\pi_1 \cup \pi_2 = \bar{\pi}_1 \cup \bar{\pi}_2$ as multisets and 
          $\pi_1 \cap \pi_2 = \bar{\pi}_1 \cap \bar{\pi}_2$ is nonempty. The second class
          consists of all binomials
        \begin{equation}
          \label{eq:bottompath}
          p_{\pi_1} p_{\pi_2} \cdots p_{\pi_s}\, -\, p_{\bar{\pi}_1} p_{\bar{\pi}_2} \cdots p_{\bar{\pi}_s}, 
        \end{equation}
        where $\pi_1,\bar{\pi}_1 , \ldots, \pi_s,\bar{\pi}_s$ are constructed as follows: Choose 
        $i \in \{0,1,\ldots,n{-}1\}$ and take any cycle 
        $\gamma = (a_0 {<} a_1 {>} a_2 {<} \cdots {<} a_{2m-1} {>} a_{2m} {=} a_0)$ in the subposet
        $\GeneralPoset_{i,i+1}$ of all elements having  rank $i$ or $i+1$ in $\GeneralPoset$.
        Then the maximal chains $\pi_j, \bar{\pi}_j$ for 
        $ 0 \leq j \leq s$ are chosen such that 
	$$ \begin{matrix}
	& \pi_j &=& (  \,u_{j,0} = \bar{u}_{j,0} < \cdots < u_{j,i} = \bar{u}_{j,i} = a_{2j} <  a_{2j+1} = u_{j,i+1}< \cdots < u_{j,n} \,)\\
	\text{and}\,\, &
	\bar{\pi}_j & =& \,(\,u_{j,0} = \bar{u}_{j,0} < \cdots < u_{j,i} = \bar{u}_{j,i} = a_{2j} <  a_{2j-1} =  \bar{u}_{j,i+1}< \cdots < \bar{u}_{j,n}\,)
	\end{matrix}
	$$
	and  the multisets $\{ u_{j,\ell} \,|\, 0 \leq j \leq s, i \leq \ell \leq n\}$   
        and $\{ \bar{u}_{j,\ell}\,  |\,
        0 \leq j \leq s, i \leq \ell \leq n\}$ coincide. 
  \end{Theorem}

  In \ref{fig:ascend} we give a visual description of the 
  binomial \ref{eq:bottompath}.

    \begin{figure}
      \begin{picture}(0,0)%
        \includegraphics{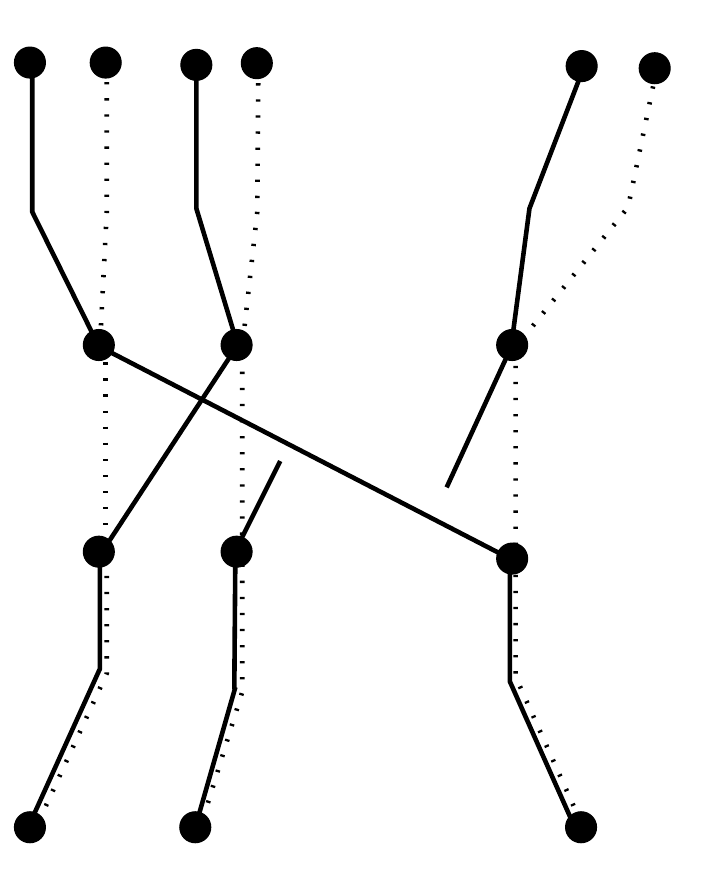}%
        \end{picture}%
        \setlength{\unitlength}{2901sp}%
        \begingroup\makeatletter\ifx\SetFigFont\undefined%
          \gdef\SetFigFont#1#2#3#4#5{%
          \reset@font\fontsize{#1}{#2pt}%
          \fontfamily{#3}\fontseries{#4}\fontshape{#5}%
          \selectfont}%
        \fi\endgroup%
        \begin{picture}(4046,5779)(1156,-6740)
          \put(1200,-1100){\makebox(0,0)[lb]{\smash{{\SetFigFont{8}{9.6}{\rmdefault}{\mddefault}{\updefault}{\color[rgb]{0,0,0}$u_{0n}$}}}}}
          \put(1771,-1100){\makebox(0,0)[lb]{\smash{{\SetFigFont{8}{9.6}{\rmdefault}{\mddefault}{\updefault}{\color[rgb]{0,0,0}$\bar{u}_{0n}$}}}}}
          \put(2300,-1100){\makebox(0,0)[lb]{\smash{{\SetFigFont{8}{9.6}{\rmdefault}{\mddefault}{\updefault}{\color[rgb]{0,0,0}$u_{1n}$}}}}}
          \put(2800,-1100){\makebox(0,0)[lb]{\smash{{\SetFigFont{8}{9.6}{\rmdefault}{\mddefault}{\updefault}{\color[rgb]{0,0,0}$\bar{u}_{1n}$}}}}}
          \put(4800,-1100){\makebox(0,0)[lb]{\smash{{\SetFigFont{8}{9.6}{\rmdefault}{\mddefault}{\updefault}{\color[rgb]{0,0,0}$u_{sn}$}}}}}
          \put(5363,-1100){\makebox(0,0)[lb]{\smash{{\SetFigFont{8}{9.6}{\rmdefault}{\mddefault}{\updefault}{\color[rgb]{0,0,0}$\bar{u}_{sn}$}}}}}
          \put(0900,-4700){\makebox(0,0)[lb]{\smash{{\SetFigFont{8}{9.6}{\rmdefault}{\mddefault}{\updefault}{\color[rgb]{0,0,0}$a_0=a_{2s}$}}}}}
          \put(2350,-4700){\makebox(0,0)[lb]{\smash{{\SetFigFont{8}{9.6}{\rmdefault}{\mddefault}{\updefault}{\color[rgb]{0,0,0}$a_2$}}}}}
          \put(4650,-4700){\makebox(0,0)[lb]{\smash{{\SetFigFont{8}{9.6}{\rmdefault}{\mddefault}{\updefault}{\color[rgb]{0,0,0}$a_{2s-2}$}}}}}
          \put(1000,-3200){\makebox(0,0)[lb]{\smash{{\SetFigFont{8}{9.6}{\rmdefault}{\mddefault}{\updefault}{\color[rgb]{0,0,0}$a_{2s-1}$}}}}}
          \put(2350,-3200){\makebox(0,0)[lb]{\smash{{\SetFigFont{8}{9.6}{\rmdefault}{\mddefault}{\updefault}{\color[rgb]{0,0,0}$a_1$}}}}}
          \put(4650,-3200){\makebox(0,0)[lb]{\smash{{\SetFigFont{8}{9.6}{\rmdefault}{\mddefault}{\updefault}{\color[rgb]{0,0,0}$a_{2s-3}$}}}}}
          \put(0900,-6700){\makebox(0,0)[lb]{\smash{{\SetFigFont{8}{9.6}{\rmdefault}{\mddefault}{\updefault}{\color[rgb]{0,0,0}$u_{00} = \bar{u}_{00}$}}}}}
          \put(2000,-6700){\makebox(0,0)[lb]{\smash{{\SetFigFont{8}{9.6}{\rmdefault}{\mddefault}{\updefault}{\color[rgb]{0,0,0}$u_{10} = \bar{u}_{10}$}}}}}
          \put(4500,-6700){\makebox(0,0)[lb]{\smash{{\SetFigFont{8}{9.6}{\rmdefault}{\mddefault}{\updefault}{\color[rgb]{0,0,0}$u_{s0} = \bar{u}_{s0}$}}}}}
      \end{picture}%
      \caption{A binomial in the Gr\"obner basis of the \Bottom model}
      \label{fig:ascend}
    \end{figure}
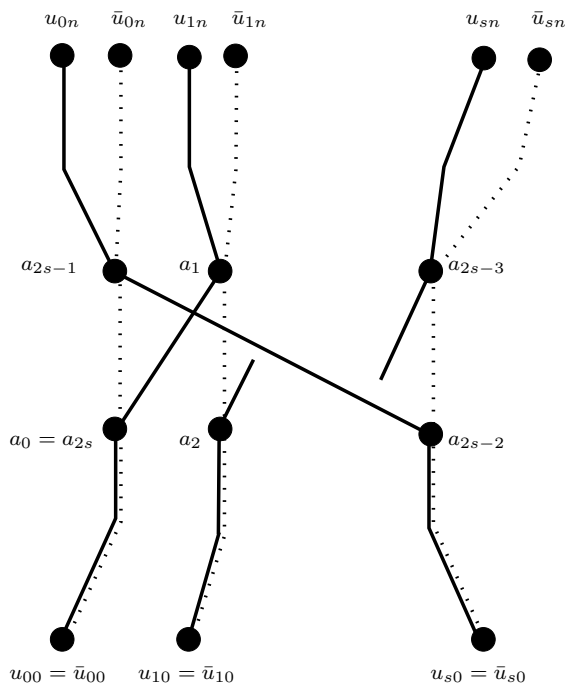
    
  For the proof of this result we shall employ Sullivant's
  theory of {\em toric fiber products} from~\cite{Sullivant2007}. We briefly review that theory. Consider two 
  polynomial rings $\KK[\,p'\,]$ and $\KK[\,p''\,]$ and a surjective multigrading $\phi : \{ p' \} \cup \{p''\} 
  \rightarrow \aA \subseteq \RR^d$, called the {\em $\aA$-grading}. 
  Then choose new variables $z_{\pi,\tau}$ for all $\pi \in \{ p'\}$ and 
  $\tau \in \{p''\}$ such
  that $\phi(\pi) = \phi(\tau)$. For ideals $I$ in $\KK[\,p'\,]$ and $J$ in 
  $\KK[\,p''\,]$ that are $\aA$-homogeneous,  we let $I \times_{\aA} J$ denote
  the kernel of the map $z_{\pi,\tau} \mapsto p'_\pi \otimes p''_\tau$
  from $\KK[\, z\,]$ to the tensor product $\KK[\,p'\,]/I \otimes \KK[\,p''\,]/J$.
   
  In order to describe a Gr\"obner basis of $I \times_{\aA} J$ in terms of Gr\"obner bases of $I$ and $J$,
  the concept of lifting monomials turns out to be crucial \cite[p. 567]{Sullivant2007}.
  A lift of a variable $p_{\pi}'$ is $z_{\pi\tau}$ for some $\tau$ with $\phi(\pi) = \phi(\tau)$. 
  Now assume that $\aA$ is linearly independent.
  Let $f \in \KK[\, p'\,]$ be an $\aA$-homogeneous  polynomial. Each monomial $m$ in $f$  factors as 
  $m_{a_1} \ldots m_{a_r}$ where $\aA = \{ a_1, \ldots , a_r\}$ and 
  $\phi(m_{a_i}) = \deg(m_{a_i}) \, a_i$. Moreover, since $\aA$ is linearly independent,
  each monomial $m$ in $f$ gives
  the same number $d_i := \deg(m_{a_i})$ of variables of degree $a_i$ (counted with multiplicity).
  Now choose a multisets of $d_i$ variables $p''$ of degree $a_i$. 
  A lift of $f$ is then any polynomial obtained from the above choices when lifting the variables
  in each monomial from $f$ in such a way that for all monomials the chosen multisets are exhausted.

  \begin{proof}
    We proceed by induction on $n = {\rm rank}(Q)$.
    If $n=1$ then \ref{eq:bottomquadric} describes an empty set of binomials and 
    the set in \ref{eq:bottompath} coincides with the Gr\"obner basis
    given in \ref{lem:bottomrank1}.

    Now assume $n \geq 2$. As in the proof of \ref{thm:bottompolytope}  
    we split  $\GeneralPoset$  into the subposet
    $\GeneralPoset' = \GeneralPoset_0 \cup \cdots \cup \GeneralPoset_{n-1}$ 
    consisting of ranks $0, \ldots, n-1$ and 
    the bipartite poset $\GeneralPoset'' = \GeneralPoset_{n-1} \cup
    \GeneralPoset_n$ consisting of ranks $n-1$ and $n$. Assume 
    $\GeneralPoset_{n-1} = \{a_1, \ldots, a_r\}$. 
    Any chain in $\MaxChain(\GeneralPoset')$ ends in an element
    from $\GeneralPoset_{n-1}$, and any chain from $\MaxChain(\GeneralPoset'')$
    starts in an element from $\GeneralPoset_{n-1}$. We consider the
    polynomial ring $\KK[\,p'\,]$ with variables $p_\pi'$ for $\pi \in 
    \MaxChain(\GeneralPoset')$ and $\KK[\,p''\,]$ with variables $p_\pi''$ for $\pi \in
    \MaxChain(\GeneralPoset'')$.
    Then we grade $p_\pi'$ by $e_i \in \RR^r$ if $\pi$ ends in $a_i$ and
    $p_c''$ by $e_i \in \RR^r$ if $\pi$ begins in $a_i$. Note that
    the set of degrees $\aA = \{e_1, \ldots,     e_r\}$
    is linearly independent.

    We write $I_{\bott}'$ for the ideal of the
    \Bottom model of $\GeneralPoset'$ and $I_{\bott}''$ for the ideal of the
    \Bottom model of $\GeneralPoset''$.
    The toric ideal of interest to us is the fiber product $\,I_{\bott} = I_{\bott}' \times_{\aA}
    I_{\bott}''$.  Since $\aA$ is linear independent, we can  apply
    \cite[Theorem 12]{Sullivant2007} and the induction hypothesis to prove the claim.
    Sullivant's result tells us that a    Gr\"obner basis of $I_{\bott}$
    can be found by lifting Gr\"obner bases of the ideals $I_{\bott}'$ and 
    $I_{\bott}''$ and by adding some quadratic relations.  

    By induction, $I_{\bott}'$ has a Gr\"obner basis $\mathcal{G}'$
    consisting of elements   \ref{eq:bottomquadric} and \ref{eq:bottompath}.
    We shall lift these to binomials in $I_{\bott}$. 
    Likewise, $I_{\bott}''$ has a Gr\"obner basis $\mathcal{G}''$
    consisting of elements \ref{eq:bottompath}.
    There are no binomials of type \ref{eq:bottomquadric} in $I_{\bott}''$
    because the poset $\GeneralPoset''$ has only rank $1$.

\smallskip

    \noindent {\sf Lifting \ref{eq:bottomquadric}:}
    Let $p_{\pi_1}  p_{\pi_2} \,-\, p_{\bar{\pi}_1}  p_{\bar{\pi}_2}$ 
    be a quadric \ref{eq:bottomquadric} in $\mathcal{G}'$.
    Since it is $\aA$-homogeneous,    the
    multisets of endpoints of $\pi_1, \pi_2$ and $\bar{\pi}_1, \bar{\pi}_2$ coincide.
    Suppose $\pi_1$ and $\bar{\pi}_1$ have the same endpoint.
    In the lifting described above we need to distinguish two cases.
        
    {\em Case 1}: $\pi_1$ and $\pi_2$ end in different endpoints. Then, 
    for any two maximal chains $\tau_1, \tau_2$ in $\GeneralPoset''$ starting in the 
    endpoints of $\pi_1$ and $\pi_2$ respectively, the unique lift for these 
    choices is 
    \begin{equation}
        \label{eq:firstlift} 
        p_{\pi_1 \tau_1} \cdot p_{\pi_2\tau_2} \,-\, p_{\bar{\pi}_1\tau_1} \cdot p_{\bar{\pi}_2 \tau_2}
        \quad \in \,\, I_{\bott}.
    \end{equation}
    
    {\em Case 2}: $\pi_1$ and $\pi_2$ end in the same endpoint. Then, for any 
    two chains $\tau_1, \tau_2$ in $\GeneralPoset''$ starting in the common endpoint 
    of $\pi_1$ and $\pi_2$, besides the lift \ref{eq:firstlift} we also have the    
    lift 
    \begin{equation}
        \label{eq:secondlift} 
        p_{\pi_1 \tau_1} \cdot p_{\pi_2 \tau_2} \,-\, p_{\bar{\pi}_1 \tau_2} \cdot p_{\bar{\pi}_2 \tau_1}
        \quad \in \,\, I_{\bott}.
    \end{equation}
    One easily checks that the binomials from \ref{eq:firstlift} and \ref{eq:secondlift} 
    satisfy the conditions from \ref{eq:bottomquadric}.

\smallskip

    \noindent {\sf Lifting \ref{eq:bottompath}:}
    First consider a       binomial 
         $p_{\pi_1} \cdots p_{\pi_s} - p_{\bar{\pi}_1} \cdots p_{\bar{\pi}_s}$ 
         of type \ref{eq:bottompath} in the Gr\"obner basis $\mathcal{G}'$. 
         Since it is 
         $\aA$-homogeneous, the multisets 
         $\{ \phi(\pi_1) , \ldots, \phi(\pi_s)\}$ and 
         $\{ \phi(\bar{\pi}_1) , \ldots, \phi(\bar{\pi}_s)\}$ coincide. Now choose 
         maximal chains $\pi_1'' , \ldots, \pi_s''$ from $\GeneralPoset''$
         with the same multiset of $\aA$-degrees 
         $\{ \phi(\pi_1'') , \ldots, \phi(\pi_s'')\}$. 
         Note that the $\pi_i''$ are just single cover relations.
         For any $\gamma \in \symm_s$ such that $\phi(\bar{\pi}_j) = \phi(\pi_{\gamma(j)}'')$, the binomial 
         $$p_{\pi_1\pi_1''} \cdots p_{\pi_s\pi_s''}\, -\, p_{\bar{\pi}_1\pi_{\tau(1)}''} \cdots p_{\bar{\pi}_s\pi_{\gamma(s)}''}$$
         lies in  $I_{\bott}$ and is of type \ref{eq:bottompath}.

We next consider
      a binomial 
         $p_{\pi_1} \cdots p_{\pi_s} - p_{\bar{\pi}_1} \cdots p_{\bar{\pi}_s}$ 
         of type \ref{eq:bottompath} in the Gr\"obner basis $\mathcal{G}''$. 
The proof is analogous to the previous case, but
                   the multiset of
         $\aA$-degree $\{ \phi(\pi_1) , \ldots, \phi(\pi_s)\} $ $= 
         \{ \phi(\bar{\pi}_1) , \ldots, \phi(\bar{\pi}_s)\}$ here is actually a set.
         Choosing a set $\{\pi_1',\ldots, \pi_s'\}$ of maximal chains
         from $\GeneralPoset'$ for which 
         $\{ \phi(\pi_1) , \ldots, \phi(\pi_s)\}$ and 
         $\{ \phi(\pi_1') , \ldots, \phi(\pi_s')\}$ coincide leads to a
         unique lift
         $$p_{\pi_1''\pi_1} \cdots p_{\pi_s''\pi_s}\, -\, p_{\pi_1''\bar{\pi}_{1}} \cdots p_{\pi_s''\bar{\pi}_{s}}$$
         is $I_{\bott}$ of type \ref{eq:bottompath}.
    All the binomials constructed by these liftings from $\mathcal{G}'$ and
    $\mathcal{G''}$ are among the
    binomials described in \ref{eq:bottomquadric} and \ref{eq:bottompath} for
    the ideal $I_{\bott}$ we seek to generate.

    Finally, we add the quadratic binomials 
    $p_{\pi_1'\pi_1''} p_{\pi_2'\pi_2''} - p_{\pi_1'\pi_2''} p_{\pi_2'\pi_1''}$ 
    for all maximal chains
    $\pi_1', \pi_2' \in \MaxChain(\GeneralPoset')$ and $\pi_1'', \pi_2'' \in \MaxChain(\GeneralPoset'')$ whose
    $\aA$-degrees coincide. These binomials lie in $I_{\bott}$ and they have  type \ref{eq:bottomquadric}.

    We have shown that the lifting of the Gr\"obner bases $\mathcal{G}'$ for $I_{\bott}'$ and
    $\mathcal{G}''$ for $I_{\bott}''$ plus the additional quadrics are
    a subset of the binomials described in  \ref{eq:bottomquadric} and \ref{eq:bottompath}.
    Using \cite[Theorem 12]{Sullivant2007}, we conclude that the binomials
    from  \ref{eq:bottomquadric} and \ref{eq:bottompath} form a Gr\"obner basis of $I_{\bott}$. 
    Actually, the following converse is true as well:
    all binomials \ref{eq:bottomquadric} and \ref{eq:bottompath} in $I_{\bott}$
    arise from $I_{\bott}'$ and $I_{\bott}''$ using the lifting procedure we described.
  \end{proof}

  \begin{Corollary}
    The toric algebra $\KK[\,p\,] / I_{\bott}$ is  normal and Cohen-Macaulay.
  \end{Corollary}

  \begin{proof}
    \ref{thm:bottomideal} gave a Gr\"obner basis for $I_{\bott}$ whose
    leading monomials are squarefree. This shows that
    $\KK[\,p\,] / I_{\bott}$ is normal. 
    Hochster's Theorem \cite[Theorem 1]{Hochster1972} implies Cohen-Macaulayness.
  \end{proof}

  We could also give an alternative proof of \ref{thm:csiszargroebner}
  using toric fiber products.
    Namely, the toric algebra $\KK[\,p\,]/I_\csi$ can be obtained as
  an iterated toric fiber product of  suitably graded smaller polynomial rings that are
  attached to the pieces in a decomposition 
  of $\GeneralPoset$ into antichains.
  The matrices $M_q$ introduced after the proof of 
  \ref{thm:csiszargroebner}  
  represent the
  ``glueing quadrics'' used for constructing
  larger toric ideals from smaller ones. 

  We close with some brief remarks on the \Bottom model  
  for the Boolean lattice $\GeneralPoset = 2^{[n]}$.
  In Section 2 we saw that, for $n=3$, the ideal $I_{\bott}$ is  principal 
  with generator
  $\,  p_{123}p_{231}p_{312} - p_{132} p_{213} p_{321} $. This cubic
  is of type \ref{eq:bottompath}. It represents the unique cycle in 
  the hexagon $Q_{1,2}$.
 
  For $n=4$, the minimal Markov basis of the \Bottom model consists of
  $6$ quadrics, $ 64$ cubics and  $ 93$ quartics.
  Thus, here we encounter binomials of both types      
  \ref{eq:bottomquadric} and    \ref{eq:bottompath}.
  The Hilbert series of the Cohen-Macaulay ring $\,\KK[\,p\,] / I_{\bott}\,$
  for $\,\GeneralPoset  = 2^{[4]}\,$ is found to be
  $$ \begin{small} \frac{1 + 12t + 72t^2  + 228t^3  + 291t^4  + 168t^5  + 36t^6}{(1-t)^{12}} .
  \end{small} $$
%  Based on this data we conjecture that the $a$-invariant of 
%  $\,\KK[\,p\,] / I_{\bott}\,$ for $\,\GeneralPoset = 2^{[n]}\,$ is $\,-n!/2$. 
  
 \section{The \Inversion Model}
  
  The inversion model is defined only in the case when 
  $\GeneralPoset $ is the distributive lattice associated
  with a constraint poset  $\PartialRanking$ on $[n]$.
  The maximal chains in $\GeneralPoset$ correspond to
  linear extensions $\pi \in \LinExt(\PartialRanking)$ of the constraint poset.
  These are the permutations $\pi \in \symm_n$ that are compatible with
  $\PartialRanking$.     Fix unknowns 
  $u_{ij}$ and $v_{ij}$ for  $1 \leq i < j \leq n$. 
  Algebraically, the \Inversion model is defined 
  by the toric ideal which is the kernel of the monomial map
  $$p_\pi \, \,\,\, \mapsto 
    \prod_{{1 \leq i < j \leq n} \atop {\pi^{-1}(i) < \pi^{-1}(j)}} \!\!\!\! u_{ij}
    \prod_{{1 \leq i < j \leq n} \atop {\pi^{-1}(i) > \pi^{-1}(j)}} \!\!\!\!\! v_{ij}.$$
  We begin considering the {\em unconstrained \Inversion model}.
  By this we mean the case when $\PartialRanking$ is an $n$-element antichain,
  so there are no constraints at all. In that unconstrained case, we have 
  $\GeneralPoset = 2^{[n]}$ 
  and our state space  
  $\MaxChain(\GeneralPoset) = \symm_n = \LinExt(\PartialRanking)$
  consists of all $n!$ permutations.

  The {\em Mallows model} \cite{Marden1995} is a
  natural specialization of the unconstrained \Inversion model to a single 
  parameter $q$.
  It is obtained by setting   $u_{ij} := 1$ and $v_{ij} :=q$. So,
  in this model, the probability of
  observing the permutation    $\pi$ is
  $P(\pi) = Z^{-1} q^{|\inv(\pi)|}$, where
    $$ \inv(\pi) \,\, = \,\, \bigl\{  (i,j)\,:\, 1 \leq i < j \leq n,\,   \pi^{-1}(i) > \pi^{-1}(j) \bigr\} $$
  is the set of inversions of $\pi$, and $Z$ is a normalizing constant.
  In contrast,
  our inversion model permits different parameters for the various inversions
  occurring in a permutation.
      
  The model polytope for the unconstrained \Inversion model  is a familiar
  object in combinatorial optimization, where it is known as
  the {\em linear ordering polytope}  
  \cite{Fiorini2006, GroetschelJuengerReinelt1985}.
  It is known that optimizing a general linear function over the linear
  ordering polytope is an NP-hard problem 
  \cite{GroetschelJuengerReinelt1985}.
  This mirrors the fact that the facial structure
  of this polytope is very complicated and a complete description appears 
  out of reach.
  As a result of this, we expect the toric rings associated with the
  \Inversion models to be more complicated than those studied in the
  previous two sections. Our study was limited to finding some computational 
  results.

  \begin{Theorem} \label{eq:justacomputation}
    For $n \leq 6$ the toric ring 
    of the unconstrained \Inversion model 
    is normal and hence Cohen-Macaulay.
    For $n \leq 5$ it is Gorenstein and its Markov basis 
    consists of quadrics.
    For $n= 6$ it is not Gorenstein and there exists
    a Markov basis element of degree~$3$.
  \end{Theorem}

\begin{proof}
  Computations using {\tt 4ti2} \cite{4ti2} show
  that the Markov basis for $n=3,4,5$ consists of $2,81,3029$ quadratic binomials.
  We do not know whether there is a quadratic Gr\"obner basis for $n=5$,
  or whether the ring is Koszul. The Hilbert series for $n \leq 5 $~are
  $$
    \begin{array}{cc}
      n & \mbox{Hilbert Series} \\
      \hline 
     3 & (1+2t+t^2)/(1-t)^4 \\ 
      4 & (1+17t+72t^2+72t^3+17t^4+t^5)/(1-t)^7 \\ 
      5 & (1{+}109t{+}2966{t}^{2}
      {+}22958{t}^{3}{+}61026{t}^{4}{+}61026{t}^{5}{+}22958{t}^{6}{+}2966{t}^{7}{+}109t^8{+}t^9)/(1-t)^{11}
    \end{array}
  $$
  All three numerator polynomials are symmetric.
  Using {\tt normaliz} \cite{Normaliz2010} one checks that the toric ring is normal in each case.
  Hochster's Theorem \cite{Hochster1972} implies that it is Cohen-Macaulay.
  The Gorenstein property  now follows from the general result
  that any Cohen Macaulay domain
  whose Hilbert series has a symmetric numerator polynomial is Gorenstein.
  
  For $n=6$, the computations are much harder, and they reveal that
  the above nice properties no longer hold. 
  The software also found
  that the Hilbert series of this unconstrained \Inversion model 
  is the product of $1/(1-t)^{16}$ and the remarkable numerator polynomial
  $$
    \begin{matrix}
      1+704\,t+117783\,t^2+5125328\,t^3+76415229\,t^4 \\+475189840\,t^5+
      1372165343\,t^6+ 1943081264\,t^7+1372165343\,t^8+
      475189840\,t^9\\ +76416069\,t^{10}+5127008\,t^{11}+118623\,t^{12}+704\,t^{14}+t^{14}.
    \end{matrix} 
  $$
  This polynomial is close to symmetric but not symmetric, so the  ring is not Gorenstein.

In addition to $130377$ quadrics, a Markov basis for $n=6$ must contain the cubic binomial
  \begin{equation}
  \label{eq:cubicmarkov}
p_{123456}p_{123645}p_{416253} \,\,-\,\, p_{123465} p_{162345}p_{412536}.
\end{equation}
Indeed, a computation shows that these are only two cubic monomials in the fiber given by
the multiset of inversions
  $\{ (1{,}4),(2{,}4),(2{,}6),(3{,}4),(3{,}5),(3{,}6),(4{,}6),(5{,}6),(5{,}6) \}$.
  \end{proof}

  A complete description of the binomial quadrics in a
  Markov basis  was recently found by
  Katth\"an \cite{Katthaen2011}.
     However, the problem of characterizing a full Markov basis 
  is widely open.

  We do not know whether  normality  holds
  for $n \geq 7$, but we suspect not.
  To address this question,
  we return to the general situation of an underlying
  constraint poset $\PartialRanking$.
  The states $\pi$ of the {\em $\PartialRanking$-constrained \Inversion model}
  are elements of the subset  $\LinExt(\PartialRanking) \subset \symm_n$.
  This inclusion corresponds to passing to some coordinate hyperplanes
  in the ambient space of the model polytopes. Therefore, the model polytope
  for the $\PartialRanking$-constrained model is a face of the model polytope for the
  unconstrained model. Hence, to answer our question about normality for $n\geq 7$,
  it could suffice to show that the toric ring for $\PartialRanking$ is not normal.

%  On the other hand, if the toric ring were normal, and hence Cohen-Macaulay,
 % then it would follow that its $a$-invariant is $-2$. Indeed, the
%  $a$-invariant of a normal toric ring is $-r$ where
 % $r$ is the smallest dilation factor for which the dilated model polytope contains 
 % an interior point.  Since the linear ordering polytope contains no interior point, we have $r \geq 2$ for this
 % polytope. But the point with the first ${n \choose 2}$ coordinates $1$ and the
  % last ${n \choose 2}$ coordinates $0$ lies in the interior of the second dilation and hence
  % the $a$-invariant is $-2$.  

  At present our state of knowledge about the $\PartialRanking$-constrained 
  \Inversion models is rather limited. We do not yet even have useful  formula
  for the dimension of its model polytope. By contrast, the dimension of the
  unconstrained model equals ${n \choose 2}$, as this is the dimension
  of the linear ordering polytope. This was shown, for example, in
  \cite[Proposition 3.10]{ReinerSaliolaWelker2010}.

  We wish to mention a family of constraint posets
  that is important for applications of statistical ranking in data mining, e.g.~in 
  recent work of Cheng {\it et al.} \cite{ChengDembczynskiHuellermeier2010}.
  For that application one would take $\PartialRanking$ 
  to be any disjoint union of a chain and an antichain.

  \begin{Example} \label{ex:mixedposet}
    Let $n \geq 4$ and $\PartialRanking$ be the poset consisting
    of the  $3$-chain $1 < 2 < 3$ and $n-3$ incomparable elements. 
    If $n=4$ then  $\LinExt(\PartialRanking) = \{1234, 1243, 1423, 4123\}$
    and the toric ideal $I_{\inv}$ is the zero ideal in the
    polynomial ring in four unknowns. If $n=5$ then the number of states is $20$
    and the model polytope has dimension $7$, degree $82$, and the Hilbert series is  
    $$
    \begin{small}
    \frac{1 + 12 t + 38 t^2 + 28 t^3 + 3 t^4}{(1-t)^8}.
    \end{small}
    $$

   The Markov basis for 
   this $\PartialRanking$-constrained model
   consists of $40$ quadrics:
   $$  
      \begin{tiny} 
        \begin{matrix}
    p_{41523}p_{51423}-p_{14523}p_{54123} & p_{41253}p_{51423}-p_{14253}p_{54123} & 
    p_{41235}p_{51423}-p_{14235}p_{54123} & p_{41253}p_{51243}-p_{12453}p_{54123} \\
    p_{41235}p_{51243}-p_{12435}p_{54123} & p_{15423}p_{51243}-p_{15243}p_{51423} & 
    p_{14253}p_{51243}-p_{12453}p_{51423} & p_{14235}p_{51243}-p_{12435}p_{51423} \\ 
    p_{41235}p_{51234}-p_{12345}p_{54123} & p_{15423}p_{51234}-p_{15234}p_{51423} & 
    p_{15243}p_{51234}-p_{15234}p_{51243} & p_{14235}p_{51234}-p_{12345}p_{51423} \\
    p_{12543}p_{51234}-p_{12534}p_{51243} & p_{12435}p_{51234}-p_{12345}p_{51243} &
    p_{15423}p_{45123}-p_{14523}p_{54123} & p_{15243}p_{45123}-p_{41523}p_{51243} \\
    p_{15234}p_{45123}-p_{41523}p_{51234} & p_{12543}p_{45123}-p_{12453}p_{54123} &
    p_{12534}p_{45123}-p_{41253}p_{51234} & p_{12354}p_{45123}-p_{12345}p_{54123} \\
    p_{15243}p_{41253}-p_{12543}p_{41523} & p_{15234}p_{41253}-p_{12534}p_{41523} &
    p_{14523}p_{41253}-p_{14253}p_{41523} & p_{15234}p_{41235}-p_{12354}p_{41523}  \\
    p_{14523}p_{41235}-p_{14235}p_{41523} & p_{14253}p_{41235}-p_{14235}p_{41253} & 
    p_{12534}p_{41235}-p_{12354}p_{41253} & p_{12453}p_{41235}-p_{12435}p_{41253} \\
    p_{14253}p_{15243}-p_{12453}p_{15423} & p_{14235}p_{15243}-p_{12435}p_{15423} & 
    p_{14235}p_{15234}-p_{12345}p_{15423} & p_{12543}p_{15234}-p_{12534}p_{15243} \\
    p_{12435}p_{15234}-p_{12345}p_{15243} & p_{12543}p_{14523}-p_{12453}p_{15423} & 
    p_{12534}p_{14523}-p_{14253}p_{15234}  & p_{12354}p_{14523}-p_{12345}p_{15423} \\
    p_{12534}p_{14235}-p_{12354}p_{14253} & p_{12453}p_{14235}-p_{12435}p_{14253} & 
    p_{12435}p_{12534}-p_{12345}p_{12543} & p_{12354}p_{12453}-p_{12345}p_{12543}
        \end{matrix} 
      \end{tiny}
  $$
  \end{Example}

  It can be asked which $\PartialRanking$-constrained \Inversion models 
  have a Markov basis of quadrics and, more generally, which degrees appear in a
  Markov basis. We confirmed the quadratic Markov basis for all posets $\PartialRanking$
  on $n \leq 4$ elements, all on $n=5$ elements arising by adding one incomparable element to 
  a poset on $4$ elements, and all unconstrained  models for $n \leq 5$.

  Interestingly, the notion of  \Inversion model changes if we
  define $i < j$ to be an inversion if $\pi(i) > \pi (j)$.
  The latter can be seen as a homogeneous Babington-Smith model from \cite{Marden1995}.
  The defining monomial map for this
  {\em alternative \Inversion model} equals
  $$ \qquad p_\pi \,\ \mapsto \prod_{{1 \leq i < j \leq n} \atop {\pi(i) < \pi(j)}} 
	  \!\! u_{ij}	  \prod_{{1 \leq i < j \leq n} \atop {\pi(i) > \pi(j)}} \!\! v_{ij}
     \qquad \hbox{for} \,\,\,\pi \in \LinExt(\PartialRanking). $$
  For the $3$-chain $1 < 2 < 3$ with two incomparable elements,
  the Markov basis now consists~of 
    $$
      \begin{tiny}
        \begin{matrix} 
           p_{15243}p_{51423} - p_{12543}p_{54123} & p_{15234}p_{51423} - p_{12534}p_{54123} & p_{15423}p_{51243} - p_{12543}p_{54123} &           p_{15234}p_{51243} - p_{12354}p_{54123} \\
            p_{12534}p_{51243} - p_{12354}p_{51423} & p_{15423}p_{51234} - p_{12534}p_{54123} &
           p_{15243}p_{51234} - p_{12354}p_{54123} & p_{15234}p_{51234} - p_{12345}p_{54123} \\
           p_{12543}p_{51234} - p_{12354}p_{51423} &           p_{12534}p_{51234} - p_{12345}p_{51423} & p_{12354}p_{51234} - p_{12345}p_{51243} & p_{12534}p_{15243} - p_{12354}p_{15423} \\
           p_{12543}p_{15234} - p_{12354}p_{15423} & p_{12534}p_{15234} - p_{12345}p_{15423} & p_{12354}p_{15234} - p_{12345}p_{15243} &
           p_{12354}p_{12534} - p_{12345}p_{12543} \\ p_{12435}p_{12453} - p_{12345}p_{12543} , 
                   \end{matrix} 
      \end{tiny}
    $$
    and
        $\,      p_{14235}p_{14253}p_{14523} - p_{12345}p_{15243}p_{15423} $, 
        and
        $\,p_{41235}p_{41253}p_{41523}p_{45123} - p_{12345}p_{51243}p_{51423}p_{54123} $.
        So,  unlike in \ref{ex:mixedposet}, this    
     Markov basis is not quadratic.
    The Hilbert series equals
    $$
    \begin{small}
    \frac{1 + 9t + 28t^2 + 51t^3 + 66t^4  + 63t^5  + 44t^6  + 21t^7  + 5t^8}{(1-t)^{11}}.
    \end{small}
    $$
  Note that, if
   $\LinExt(\PartialRanking)$ is closed under taking
  inversions, then this model coincides with the normal $\PartialRanking$-constraint
  \Inversion model up to a relabeling. 
    This holds for the
  unconstrained \Inversion model. 
  All examples tested in this alternative model had normal model polytopes. 
  
\section{Plackett-Luce Model and  Bradley-Terry model}
    \label{sec:plackettluce}

  The \PlackettLuce model is a non-toric model on
  the set $ \LinExt(\PartialRanking)$ of permutations $\pi \in \symm_n$ that are consistent
  with a given constraint poset $\PartialRanking$ on $[n]$.
  It can be defined by the map
  \begin{equation} \label{PLmap2}
    p_\pi \,\, \mapsto \,\,\prod_{i=1}^{n-1} \frac{1}{ \sum_{j=1}^i \theta_{\pi(j)} }
    \qquad \hbox{for $\,\pi \in        \LinExt(\PartialRanking)  $}.
  \end{equation}
  We denote this model by $\PlackettLuceShort_\PartialRanking$
  and its homogeneous ideal by $I_{\pl_\PartialRanking}$.
  Thus $I_{\pl_\PartialRanking}$ is the kernel of the ring map
  $\RR[ \,p_\pi :\pi \in  \LinExt(\PartialRanking) \,] 
   \rightarrow \RR(\theta_1,\theta_2,\ldots,\theta_n)$
  defined by the formula \ref{PLmap2}.
  The formula shows that the \PlackettLuce model
  is a submodel of the \Bottom model on $ \LinExt(\PartialRanking)$.
  In fact, the \Bottom model is the {\em toric closure} of the
  \PlackettLuce model, by which we mean that
  $\BottomPol_\PartialRanking$    is the smallest toric model containing 
  $\PlackettLuceShort_\PartialRanking$. The specialization map is
  \begin{equation} \label{PLspec} 
    t_{\pi(\{1,2 ,\ldots, i\})} \quad \mapsto \quad \bigl( \theta_{\pi(1)} +
    \theta_{\pi(2)} + \cdots + \theta_{\pi(i)} \bigr)^{-1}. 
  \end{equation}
  We fix $\KK = \CC$ and regard the \PlackettLuce model 
  $\PlackettLuceShort_\PartialRanking$ as a projective variety in
  $\PP^{|\LinExt(\PartialRanking)|-1}$.
  The toric closure property means that all binomials
  in  $I_{\pl_\PartialRanking}$ must lie in $I_{\rm asc}$,
  and this follows from unique factorization in $\mathbb{R}[\theta_1,\ldots,\theta_n]$,
  given that the linear forms in \ref{PLspec} are distinct.
    
  In order for $\PlackettLuceShort_\PartialRanking$
  to be properly defined as a statistical model, its probabilities
  should sum to $1$. For this we would need to identify the normalizing constant,
  which is the image of $\sum_{\pi \in \LinExt(\PartialRanking)} p_\pi$
  under the map  \ref{PLmap2}. A formula for this quantity can be derived,
  for many situations of interest, from equations  (25) and (26) in 
  Hunter's article \cite{Hunter2004}. The most general situation where the
  normalizing constant was determined can be found in \cite{BoussicaultFerayLascouxReiner2010}.
  They make use of sophisticated methods from the algebraic and
  geometric theory of valuations on cones.  
  In our situation, $\,\sum_{\pi \in \symm_n} p_\pi\,$ is mapped to 
  $\frac{1}{\theta_1 \theta_2\cdots \theta_n}$
  under the ring map in \ref{PLmap2}. 
 
  Let us begin by examining the unconstrained case when
  $\PartialRanking$ is an antichain, $\GeneralPoset = 2^{[n]}$ and	 
  $\LinExt(\PartialRanking) =  \MaxChain(\GeneralPoset)= \symm_n$.
  This is the  Plackett-Luce model $\PlackettLuceShort_n$ familiar from
  the statistics literature \cite{Hunter2004, Luce1959, Plackett1968}.
  With the correct normalizing constant, its parametrization equals
  \begin{equation} \label{PLmap}
    p_\pi \,\, \mapsto \,\,
        \prod_{i=1}^n \frac{\theta_{\pi(i)}}{ \sum_{j=1}^i \theta_{\pi(j)} }
            \qquad \hbox{for $\,\pi \in \symm_n$}.
  \end{equation}
  This defines a polynomial map
  from the non-negative orthant $\RR_{\geq 0}^n$ to the
  $(n! - 1)$-dimensional simplex of probability distributions on
  the symmetric group $\symm_n$. 
  We shall regard
  $\PlackettLuceShort_n$ as a complex projective variety in the
  ambient $\,\PP^{n ! - 1}$.
  Being the image of a rational map from $\PP^{n-1}$, 
  the dimension of this variety is $\,\leq n-1$.
\ref{thm:PLmain} shows that it equals $ n-1$.
    
  \begin{Example}[$n=3$] \label{exPL3} 
    The \PlackettLuce model $\PlackettLuceShort_3$ is a surface of degree $7$
    embedded in $5$-dimensional projective space $\PP^5$.
    The parameterization \ref{PLmap2} of that surface is equivalent~to
    $$
    \begin{matrix}
        p_{123} \mapsto 
          \theta_2 \theta_3 (\theta_1 {+}\theta_3 ) (\theta_2 {+}\theta_3 ), &
        p_{132} \mapsto \theta_2 \theta_3 (\theta_1 {+}\theta_2) (\theta_2 {+}\theta_3), &
        p_{213} \mapsto \theta_1 \theta_3 (\theta_1 {+}\theta_3) (\theta_2 {+}\theta_3) ,\\
        p_{231} \mapsto \theta_1 \theta_3 (\theta_1 {+}\theta_2) (\theta_1 {+}\theta_3), &
        p_{312} \mapsto \theta_1 \theta_2 (\theta_1 {+}\theta_2) (\theta_2 {+}\theta_3) ,&
        p_{321} \mapsto \theta_1 \theta_2 (\theta_1 {+}\theta_2) (\theta_1 {+} \theta_3) .
      \end{matrix}
    $$
    The defining ideal $I_{\pl_3}$ of $\PlackettLuceShort_3$ is minimally generated 
    by three quadratic polynomials, in addition to the familiar cubic binomial that 
    specifies the ambient \Bottom model:
    $$I_{\pl_3}  \,\,\, = \,\,\, 
      \bigg\langle
      \begin{matrix}
        p_{123}(p_{321} + p_{231})-p_{213}(p_{132} + p_{312}),\,
        p_{312}(p_{123} + p_{213})-p_{132}(p_{231} + p_{321}), \\
        p_{231}(p_{132} + p_{312})-p_{321}(p_{123} + p_{213}), \quad\,
        p_{123}p_{231}p_{312}-p_{132}p_{321}p_{213} \quad
      \end{matrix} \bigg\rangle.
    $$
    The singular locus of $\PlackettLuceShort_3$ consists of the three isolated points
    $e_{321}-e_{231}$, $e_{123}-e_{213}$ and  $e_{132}-e_{312}$ in $\PP^5$.
    In particular, there are no singular points with non-negative coordinates, so
    this statistical model is a smooth surface in the $5$-dimensional probability simplex.

    From the point of view of algebraic geometry, our parametrization map represents the blow-up
    of the projective plane $\PP^2 $ at the following configuration of nine special points:
    \begin{equation}
    \label{eq:ninepoints}
      \begin{matrix}
        (\,0 :0 :1\,)  & (\,0 :1 :0\,) & (\,1 :0 :0\,)
         \\ (1 :-1 :0) & (1 :0 :-1) & (0 :1:-1) \\
        (1 :1 :-1) & (1 :-1 :1) & (-1 :1 :1)
      \end{matrix} 
    \end{equation}
    This configuration has three $4$-point lines and four $3$-point lines. The map blows
    down the three $4$-point lines, and this creates 
    a rational surface in $\PP^5$ with three singular points.

    From the point of view of commutative algebra, one might ask whether
    the four generators of the ideal $I_{\pl_3}$ form a Gr\"obner basis 
    with respect to some term order.
    A computation reveals that this is not the case. However, 
    we do get a square-free Gr\"obner basis for the lexicographic
    term order with 
    $\, p_{123} {>}  p_{132} {>}  p_{213} {>}  p_{231} {>}  p_{312} {>} p_{321} $.
    The initial ideal equals
    $$ 
      \begin{matrix} 
        {\rm in}_{\rm lex} (I_{\pl_3})  \, &= &
        \langle p_{123}, p_{132}, p_{231} \rangle & \! \cap \! &
        \langle p_{123}, p_{132}, p_{312} \rangle & \! \cap \! &
        \langle p_{123}, p_{132}, p_{213} \rangle & \! \cap \! & & \\ & & 
        \langle p_{123}, p_{213}, p_{231}  \rangle & \! \cap \! & 
        \langle p_{123}, p_{213}, p_{312} \rangle & \! \cap \! &
        \langle  p_{123}, p_{312}, p_{321} \rangle & \! \cap \! &
        \langle  p_{231}, p_{312},  p_{321} \rangle. 
      \end{matrix} 
    $$
    This represents a simplicial complex of seven triangles, listed in a shelling order,
    so $I_{\pl_3}$ is Cohen-Macaulay. 
    The Hilbert series  of the  ring  $\RR[p]/I_{\pl_3}$ equals $\,(1+3t+3t^2)/(1-t)^3$. \qed
  \end{Example}

  \begin{Example}[$n=4$] \label{exPL4}
     The \PlackettLuce model ${\rm PL}_4$ is a threefold of degree
     $191$ in $\PP^{23}$. It is obtained from $\PP^3$ by blowing up
     $55$ lines.
     The  homogeneous prime ideal $I_{\pl_4}$ that defines $\PlackettLuceShort_4$ is minimally generated by
     $105$ quadrics and $75$ cubics. Its Hilbert series equals
     $$ \frac{1+20t+105t^2+65t^3}{(1-t)^4} .    $$ 
     We do not know whether
     $I_{\pl_n}$ is generated in degree $2$ and $3$ for $n \geq 5$.     \qed
  \end{Example}

  Let us now turn to the general \PlackettLuce model 
  with a given constraint poset $\PartialRanking$, so only
  permutations $\pi$ in  $\LinExt(\PartialRanking)$ are allowed.
  The model $\PlackettLuceShort_{\PartialRanking}$ is obtained from
  $\PlackettLuceShort_n$ by projecting onto those coordinates. Algebraically,
  the prime ideal $I_{\PartialRanking}$ is obtained from $I_{\pl_n}$
  by eliminating all unknowns $p_\pi$ where $\pi$ is a
  permutation that is not compatible with $\PartialRanking$.

  \begin{Example} \label{ex:sixmaxchains}
    Let $n=4$ and let $\PartialRanking$ be the poset with
    two covering relations $1 {<} 2$ and $3 {<} 4$. The corresponding
    distributive lattice $\LinExt(\PartialRanking)$ is the product of two chains
    of length $3$. Note that $\LinExt(\PartialRanking)$ has six maximal chains,
    namely, the permutations that respect $1 < 2$ and $3 < 4$.
    The corresponding unknowns are mapped to products
    of four linear forms as follows:
    $$     
      \begin{matrix}
         p_{1234} \mapsto \theta_3 (\theta_1+\theta_3) (\theta_3+\theta_4) (\theta_1+\theta_3+\theta_4) ,&
         p_{1324} \mapsto \theta_3 (\theta_1+\theta_2) (\theta_3+\theta_4) (\theta_1+\theta_3+\theta_4) ,\\
         p_{1342} \mapsto \theta_3 (\theta_1+\theta_2) (\theta_3+\theta_4) (\theta_1+\theta_2+\theta_3), & 
         p_{3124}  \mapsto \theta_1 (\theta_1+\theta_2) (\theta_3+\theta_4) (\theta_1+\theta_3+\theta_4), \\
         p_{3142} \mapsto \theta_1 (\theta_1+\theta_2) (\theta_3+\theta_4) (\theta_1+\theta_2+\theta_3), &
         p_{3412} \mapsto \theta_1 (\theta_1+\theta_2) (\theta_1+\theta_3) (\theta_1+\theta_2+\theta_3).
      \end{matrix}
    $$
    These reducible quartics meet in nine lines in $\PP^3$, so the
    parametrization of $\PlackettLuceShort_{\PartialRanking}$ blows these up.
    The ideal $I_{\PartialRanking}$ is complete intersection. Its minimal generators are the
    cubic
    $$
      p_{1234} p_{1342} p_{3142} + p_{1234} p_{3142}^2 + p_{1234} p_{3142} p_{3412}
      -p_{1234} p_{1324} p_{3412} - p_{1324}^2 p_{3412} - p_{1324} p_{3124} p_{3412} 
    $$
    and the binomial quadric $\,p_{1342}p_{3124}-p_{1324} p_{3142}\,$
    that defines the \Bottom model on $\mathcal{P}$. \qed
  \end{Example}

  The following is our main result in this section.  It should be useful for obtaining
  information about the $(n{-}1)$-dimensional variety 
  $\PlackettLuceShort_{\PartialRanking}$ and its homogeneous prime ideal $I_{\PartialRanking}$.

  \begin{Theorem} \label{thm:PLmain}
    The parameterization
    $\,\PP^{n-1} \rightarrow \PlackettLuceShort_{\PartialRanking} \subset \PP^{|\LinExt(\PartialRanking)|-1}\,$
    of the \PlackettLuce model on the poset $\PartialRanking$ is given geometrically
    as the blowing up of $\PP^{n-1}$ along an arrangement of linear subspaces of codimension $2$.
    These subspaces are defined by the equations $\sum_{i \in A} \theta_i = \sum_{j \in B} \theta_j = 0$
    where $\{A, B\}$ runs over all incomparable pairs in the distributive lattice~on~$\PartialRanking$.
  \end{Theorem}
 
  \begin{proof}
    Let $\RR[t]$ denote the polynomial ring of parameters
    in the \Bottom model \ref{eq:tmap}. Its indeterminates are $t_A$ where $A$ runs over  
    subsets of $[n]$ that are order ideals in $\PartialRanking$. We define
    $M$ to be the Stanley-Reisner ideal of the
    distributive lattice of order ideals in $\PartialRanking$.
    This is the ideal in $\RR[p]$ generated by products
    $\, t_A t_B \,$ where $A$ and $B$ are incomparable,
    meaning that neither $A \subset B$ nor $B \subset A$ holds.
    The {\em Alexander dual} of $M$ is the monomial ideal
    $$M^* \,\,\, = \,\, \bigcap_{\{A,B\}} \langle \,t_A \, , \, t_B \, \rangle , $$
    where the intersection is over all incomparable pairs of order ideals.
    The generators of $M^*$ correspond to the associated primes of $M$, so
    they are indexed by compatible permutations $\pi \in \LinExt(\PartialRanking)$.
    Interpreting $\pi$ as a maximal chain of order ideals,
    that correspondence is
    \begin{equation}
      \label{eq:maptp1}
      p_\pi \,\, \mapsto \,\, \prod_{A \not\in \pi} t_A 
       \qquad \qquad \hbox{for $\,\pi \in        \LinExt(\PartialRanking)  $}.
    \end{equation}
    The arrangement of subspaces described in the statement of
    \ref{thm:PLmain} is the intersection of the variety of $M^*$ with a subspace
    $\PP^{n-1}$ defined by $\, t_A \,= \, \sum_{i \in A} \theta_i$.
    By substituting this into \ref{eq:maptp1} we see that
    the blow-up along that subspace arrangement is defined by the map
    \begin{equation}
      \label{eq:maptp2}
      p_\pi \,\,\,\, \mapsto \,\, 
        \prod_{A \not\in \pi} \bigl( \sum_{i \in A} \theta_i \bigr) \,\, = \,\,
        {\rm const} \cdot \prod_{A \in \pi} \frac{1}{ \sum_{i \in A} \theta_i}
        \qquad \hbox{for $\,\pi \in        \LinExt(\PartialRanking)  $}.
    \end{equation}
    This is precisely the defining parametrization \ref{PLmap2} of the 
    \PlackettLuce model $\PlackettLuceShort_\PartialRanking$.
  \end{proof}
 
  \begin{Example}
    Let $n=4$ and $\PartialRanking$ as in \ref{ex:sixmaxchains}.
    Then the above Stanley-Reisner ideal is
    $$      M \,\,\, = \,\,\, \langle\,
      t_{1} t_{3},\, t_3 t_{12}, \,
      t_{12} t_{13} ,\, t_1 t_{34} ,\, t_{12} t_{34},\, t_{13} t_{34}, \,
      t_{34} t_{123},\,  t_{12} t_{134},\, t_{123} t_{134}\, \rangle.
    $$
    Its Alexander dual reveals the combinatorial pattern of the map in
     \ref{ex:sixmaxchains}:
    $$
      M^* \,\,\, = \,\,\, \langle\,
      t_3 t_{13} t_{34}  t_{134},\,
      t_3 t_{12} t_{34} t_{123},\,
      t_1 t_{12} t_{34} t_{123},\,
      t_3 t_{12} t_{34} t_{134},\,
      t_1 t_{12} t_{34}  t_{134},\,
      t_1 t_{12} t_{13} t_{123} \,\rangle.
    $$
    The model $\PlackettLuceShort_\PartialRanking$ 
    is the blow-up of $\PP^3$ at nine lines,
    one for each of the generators of $M$. \qed
  \end{Example}
    
  Each of our unconstrained ranking models was considered as a subvariety
  of the complex projective space $\mathbb{P}^{n ! - 1}$. If $K$ is any $k$-element
  subset of $[n]$ then we obtain a natural rational
  map $\PP^{n! - 1} \dashrightarrow \PP^{k! - 1}$
  which records the probabilities for each of the $\,k !\,$ orderings of $K$ only.
  Statistically, this map corresponds to {\em marginalization} for the induced orderings on $K$.
  We can now take the direct product of all of these maps, where $K$ runs over all
  $\binom{n}{k}$ subsets of cardinality $k$ in $[n]$.
  The resulting rational map into a product of projective spaces,
  \begin{equation}
    \label{from5to111}
  \PP^{n! - 1} \,\dashrightarrow \, (\PP^{k !- 1})^{\binom{n}{k}} ,
  \end{equation}
   is called the {\em complete  marginalization map of order $k$}.
  For example, if $n=3$ and $k=2$ then we are mapping into a
  product of three projective lines, with coordinates
  $(q_{12}:q_{21})$, $(q_{13}:q_{31})$ and $(q_{23}:q_{32})$ respectively.
  Here,  the complete marginalization is the rational map
  $\,    \PP^5 \,\dashrightarrow \,    \PP^1 \times \PP^1 \times \PP^1  \,$
  which is given in coordinates as follows:
  $$ 
    \begin{matrix} 
	  (q_{12}:q_{21}) & = &(p_{123}+p_{132}+p_{312} : p_{213}+p_{231}+p_{321}) \,,\\
      (q_{13}:q_{31}) &= & (p_{132}+p_{123}+p_{213}: p_{312}+p_{321}+p_{231})\,,\\
      (q_{23}:q_{32}) & = &(p_{123}+p_{213}+p_{231}: p_{132}+p_{312}+p_{321}).
    \end{matrix} 
  $$
  We shall refer to the complete marginalization of order $2$ as the
  {\em pairwise marginalization}.

  \begin{Example} \label{BTdrei}
    The pairwise marginalization of  the \PlackettLuce surface  $\PlackettLuceShort_3 \subset \PP^5$
    is the surface in $\PP^1 \times \PP^1 \times \PP^1$ that is
    defined  by the binomial equation
    $\,  q_{12} q_{23} q_{31} \, = \, q_{21} q_{32} q_{13}$. 
    The composition of the map in \ref{exPL3}
    with the map in \ref{from5to111} is a toric rational map
    $\PP^2 \dashrightarrow \PP^1 \times \PP^1 \times \PP^1$
    that blows up the three coordinate points 
    $(1{:}0{:}0)$, $(0{:}1{:}0)$ and $(0{:}0{:}1)$.
    \qed
  \end{Example}
 
  It is worthwhile, both algebraically and statistically,  to study the various marginalizations of  
  the \Csiszar model,  \Bottom model, the \Inversion model   and the \PlackettLuce model. 
  Of particular interest
  is the pairwise marginalization of the Plackett-Luce model. This is known in
  the literature as the {\em \BradleyTerry model} \cite{Hunter2004}.
  All of these marginalized models make sense relative to a fixed
  constraint poset $\PartialRanking$. Here, we regard each $k$-set $K$ as subposet
  of $\PartialRanking$ and we write the corresponding marginalization map as
  \begin{equation}
    \label{plpk} 
    \PP^{|\LinExt(\PartialRanking)|-1}\, \dashrightarrow \, \PP^{|\LinExt(K)|-1}. 
  \end{equation}
  The complete $k$-th marginalization is the image of the direct product of these maps,
  as $K$ runs over all $k$-sets. For convenience, we shall here remove those $k$-sets
  $K$ that are totally ordered in $\mathcal{P}$ because the corresponding maps
  in \ref{plpk} are constant when $|\LinExt(K)| =  1$.

  We conclude this article with  the following algebraic characterization of the \BradleyTerry model.
  We write $\PartialRanking^c$ for the bidirected graph on
  $[n]$ where $(i,j)$ is a directed edge if
  $i$ and $j$ are incomparable in $\PartialRanking$.
  Each circuit $i_1,i_2,\ldots,i_r , i_1$ in $\PartialRanking^c$
  is encoded as a binomial:
  \begin{equation}
    \label{circuits} 
    q_{i_1i_2} q_{i_2 i_3} \cdots q_{i_{r-1} i_r} q_{i_r i_1} \,-\,
    q_{i_2i_1} q_{i_3 i_2} \cdots q_{i_r i_{r-1}} q_{i_1 i_r} . 
  \end{equation}
  These binomials define hypersurfaces in $\PP^{\binom{n}{2}}$.
  For instance, the model in \ref{BTdrei}  is the toric hypersurface
  in $\PP^1 \times \PP^1 \times \PP^1$
  thus associated to a $3$-cycle. 
  
  The theorem below refers to  {\em unimodular Lawrence ideals}.
  This class of toric ideals   was introduced and studied by Bayer {\it et~al.} in \cite{BPS2001}. 
  The  associated  toric varieties  live naturally in
  a product of projective lines $\PP^1 \times \cdots \times \PP^1$. The case of interest here is
  that of unimodular Lawrence ideals arising from graphs.
  For these  ideals and their syzygies we refer to \cite[\S 5]{BPS2001}.

  \begin{Theorem} 
    The \BradleyTerry model with constraints $\PartialRanking$ is  toric. It is defined by
    the unimodular Lawrence ideal whose generators
    are the circuits \ref{circuits} in the bidirected graph~$\PartialRanking^c$.
  \end{Theorem}

  From this result we can now determine the commutative algebra invariants
  of the \BradleyTerry model, such as its Hilbert series in the $\ZZ^n$-grading and its multidegree.

  \begin{proof}
    Following \cite{Hunter2004},
    the parametrization of the \BradleyTerry model can be written as
    \begin{equation} \label{eq:BTmap}
       \qquad q_{ij} \,\, \mapsto \,\, \frac{\theta_j}{\theta_i + \theta_j} 
       \qquad \text{for $i,j$ incomparable in $\PartialRanking$.} 
    \end{equation}
    Let $\rho_{\{i,j\}}$ be new unknowns indexed by unordered pairs $\{i,j\} \subset [n]$.
    The unimodular Lawrence ideal associated with the bidirected graph
    $\PartialRanking^c$ is the kernel of the monomial~map
    \begin{equation} \label{eq:BTmap2}
      \qquad q_{ij} \,\, \mapsto \,\, \rho_{\{i,j\}} \cdot \theta_j
      \qquad \text{for $i,j$ incomparable in $\PartialRanking$.} 
    \end{equation}
    The specialization $\,\rho_{\{i,j\}} = (\theta_i + \theta_j)^{-1}$
    shows that the ideal $I_{\bt_\PartialRanking}$ of the \BradleyTerry model is
    contained the unimodular Lawrence ideal generated by the circuits \ref{circuits}.
    In addition, the ideal $I_{\bt_\PartialRanking}$ contains the
    linear polynomials $q_{ij} + q_{ji}-1$. These represent the fact that, in
    any compatible ranking $\pi$, either item $i$ ranks before item $j$ or vice versa, 
    but not both.

    Let $J$ be the ideal generated by the circuits \ref{circuits} and these linear polynomials.
    We have seen that $J \subseteq I_{\bt_\PartialRanking} $, and we are
    claiming that equality holds. But this follows by observing that both ideals
    are prime, and their varieties have the same dimension, namely $n-1$. 
    Indeed, $I_{\bt_\PartialRanking} $ is prime by definition, and
    $J$ is prime because adding the linear forms $q_{ij}+q_{ji}-1$
    to the unimodular Lawrence ideal
    simply amounts to   dehomogenizing from $\mathbb{P}^1$ to $\mathbb{A}^1$
    in each factor. Geometrically, this operation preserves the dimension of the variety.
  \end{proof}

\section*{Acknowledgments}
  We very grateful to Winfried Bruns and Raymond Hemmecke for their substantial   help 
  with the computational results  in \ref{eq:justacomputation}.
  Using the developers' versions of  {\tt Normaliz} \cite{Normaliz2010} and 
   {\tt 4ti2} \cite{4ti2} respectively, they succeeded in computing the
  Hilbert series of the \Inversion model for $n=6$
  and in finding the cubic Markov basis element \ref{eq:cubicmarkov}.
  We also thank Eyke H\"ullermeier and Seth Sullivant for helpful conversations
  and the referees for many suggestions that helped us to improve the paper.
  Bernd Sturmfels was partially supported by
  the U.S.~National Science Foundation 
  (DMS-0757207 and DMS-0968882).
  Volkmar Welker was partially supported by MSRI Berkeley.


\begin{thebibliography}{xxx}
%  \bibitem{Athanasiadis2005}
%    C.A. Athanasiadis:
%    Ehrhart polynomials, simplicial polytopes, magic squares and a conjecture of Stanley,
%    {\it J. Reine Angew. Math.} {\bf 583} (2005) 163-174.
  \bibitem{Barvinok2002}
     A.~Barvinok: {\em
     A Course in Convexity},
      Graduate Studies in Mathematics, {\bf 54},
     AMS, Providence, 2002.
  \bibitem{BoussicaultFerayLascouxReiner2010}
    A.~Boussicault, V.~Feray, A.~Lascoux and V. Reiner:
    Linear extension sums as valuations of cones,
    {\tt arXiv:1008.3278}.
  \bibitem{BPS2001}
     D.~Bayer, S.~Popescu and B.~Sturmfels:
     Syzygies of unimodular Lawrence ideals,
     {\em J. Reine Angew.~Math.}      {\bf 534} (2001) 169--186.
  \bibitem{BeerenwinkelErikssonSturmfels2006}
    N.~Beerenwinkel, N.~Eriksson and B.~Sturmfels:
    Evolution on distributive lattices,   {\em Journal of Theoretical Biology} {\bf 242} (2006) 409--420.
\bibitem{Normaliz2010}
    W.~Bruns, B.~Ichim and C.~S\"oger:
    {\em Normaliz} -- software for affine monoids, vector configurations, lattice polytopes, 
    and rational cones,
    {\tt http://www.mathematik.uni-osnabrueck.de/normaliz/},
    2010.
  \bibitem{CanfieldMcKay2009}
    E.R.~Canfield and B.D.~McKay:
    The asymptotic volume of the Birkhoff polytope,
    {\em  J.~Analytic Comb.} {\bf 4} (2009) article \#2.
  \bibitem{ChanRobbinsYuen1999}
    C.S.~Chan, D.P.~Robbins and D.S.~Yuen: 
    On the volume of a certain polytope,
    {\em Experiment. Math.} {\bf 9} (2000) 91--99. 
  \bibitem{ChengDembczynskiHuellermeier2010}
    W.~Cheng, K.~Dembczynski and E.~H\"ullermeier:
    Label ranking based on the Placket-Luce model,
    {\em Proc. ICML-2010, International Conference on Machine Learning},
    Haifa, Israel, June 2010.
 % \bibitem{Critchlow1980} D. E. Critchlow:
  %  Metric Methods for Analyzing Partially Ranked Data, 
  %  {\em Lecture Notes in Statist.} {\bf 34} Springer-Verlag, Berlin, 1980.
  \bibitem{Csiszar2009a} 
    V.~Csisz\'ar:
    Markov bases of conditional independence models for 
    permutations,
    {\em Kybernetica} {\bf 45} (2009) 249-260.
  \bibitem{Csiszar2009b} 
    V.~Csisz\'ar:
    On L-decomposability of random permutations,
    {\em J. Math. Psychology} {\bf 53} (2009) 294-297.
  %\bibitem{DeLoearaLiuYoshida2009} 
   % J.~DeLoera, F.~Liu and R.~Yoshida:
   % A generating function for all magic squares and the volume of the Birkhoff polytope,
   % {\em  J.~Alg. Comb.} {\bf 30} (2009) 113--139.
%  \bibitem{Diaconis1988}    
  %  P. Diaconis: 
 %   Group representations in probability and statistics,
  %  {\em IMS Lecture Notes-Monograph Series} {\bf 11}, Hayward, CA, Institute of Mathematical Statistics 
  %  (1988).
  \bibitem{DiaconisEriksson2006}
    P.~Diaconis and N.~Eriksson:
    Markov bases for noncommutative Fourier analysis of ranked data,
    {\em  J.~of Symbolic Computation} {\bf 41} (2006) 182--195.
  \bibitem{DiaconisSturmfels1998}
    P.~Diaconis and B.~Sturmfels:
    Algebraic algorithms for sampling from conditional distributions,
    {\em Ann. Stat.} {\bf 26} (1998) 363-397.
  \bibitem{DrtonSturmfelsSullivant2009}
    M.~Drton, B.~Sturmfels and S.~Sullivant:
    {\em Lectures on Algebraic Statistics},
    Oberwolfach Seminars, Vol 39, Birkh\"auser, Basel, 2009. 
  \bibitem{FienbergPetrovicRinaldo2010}
    S.E.~Fienberg, S.~Petrovi\'c and A.~Rinaldo:
    Algebraic statistics for a directed random graph model with reciprocation,
%    In: M.~Viana and H.~ Wynn (eds.), 
    {\em Algebraic Methods in Statistics and Probability II}, pp. 261--283, 
    Contemporary Math. {\bf 516}, Amer. Math. Soc., Providence, 2010.
  \bibitem{Fiorini2006}
    S.~Fiorini:
    $\{0, \frac 12\}$-cuts and the linear ordering problem: surfaces that define facets,
    {\em SIAM J. Discrete Math.} {\bf 20} (2006), 893--912.
  \bibitem{4ti2} 
    4ti2 team:
    {\em 4ti2} -- A software package for algebraic, geometric and 
              combinatorial problems in linear spaces,
    available at {\tt www.4ti2.de}.
  \bibitem{GeigerMeekSturmfels2006}
    D. Geiger, C. Meek and B. Sturmfels:
    On the toric algebra of graphical models,
    {\em Ann. Statist.} {\bf 34} (2006), 1463-1492. 
  \bibitem{GroetschelJuengerReinelt1985}
    M.~Gr{\"o}tschel, M.~J{\"u}nger and G.~Reinelt:
    Facets of the linear ordering polytope,
    {\em Math. Program.} {\bf 33} (1985), 43--60.
  \bibitem{Hochster1972}
    M.~Hochster:
    Rings of invariants, Cohen-Macaulay rings generated by monomials, and
    polytopes,
    {\em Ann. Math.} {\bf 96} (1972) 318--338.
  \bibitem{HommelBretzMaurer2007}
    G.~Hommel, F.~Bretz and W.~Maurer:
    Powerful short-cuts for multiple testing procedures with special
    reference to gatekeeping strategies,
    {\em Statist. Med.} {\bf 26} (2007) 4063-4073.
  \bibitem{Hunter2004} 
    D.R.~Hunter:
    MM algorithms for generalized Bradley-Terry models,
    {\em  Ann. Stat.} {\bf 32} (2004) 384-406.
  \bibitem{KatTho2007}
    A.~Katsabekis and A.~Thoma:
    Parametrizations of toric varieties over any field, 
    {\em J. Algebra} {\bf 308} (2007) 751--763. 
\bibitem{Katthaen2011}
   L.~Katth\"an:
   Decomposing sets of inversions,
   {\tt arXiv:1111.3419}.
 % \bibitem{Lenz2008}   M.~Lenz:  
%    Toric ideals of flow polytopes,    preprint,   {\tt arXiv:0801.0495}.
  \bibitem{Luce1959} 
    R.D.~Luce: 
    {\em Individual Choice Behavior},
    Wiley, New York, 1959.
  \bibitem{Marden1995} 
    J.I.~Marden:
    {\em Analyzing and Modeling Rank Data}, 
    Monographs on Statistics and Applied Probability, {\bf 64},
    Chapman \& Hall, London, 1995
  \bibitem{OhsugiHibi1998} 
    H.~Ohsugi and T.~Hibi: 
    Normal polytopes arising from finite graphs,
    {\em J. Algebra} {\bf 207} (1998) 409--426.
  \bibitem{OhsugiHibi1999} 
    H.~Ohsugi and T.~Hibi: 
    Toric ideals generated by quadratic binomials, 
    {\em J. Algebra} {\bf 218} (1999) 509--527.
  \bibitem{PachterSturmfels2005} 
    L.~Pachter and B.~Sturmfels:
    {\em Algebraic Statistics for Computational Biology},
    Cambridge University Press, Cambridge, 2005.
   \bibitem{Plackett1968} 
    R.L. Plackett:
    Random permutations. 
    {\em J. R. Stat. Soc., Ser. B} {\bf 30} (1968) 517--534.
  \bibitem{ReinerSaliolaWelker2010}
    V.~Reiner, F.~Saliola and V.~Welker:
    Spectra of symmetrized shuffling operators,
    {\tt arXiv:1102.2460}.
  \bibitem{Sturmfels1996} 
    B.~Sturmfels:
    Gr{\"o}bner bases and convex polytopes,
    Univ. Lect. Ser. {\bf 8}, 
    AMS, Providence, 1996. 
  \bibitem{Sullivant2007}
    S.~Sullivant:
    Toric fiber products, 
    {\it J. Algebra} {\bf 316} 
    (2007) 560--577.
  \bibitem{Villerreal2001}
    R.H.~Villarreal:
    {\em Monomial Algebras},
    Pure and Appl. Math. {\bf 238},
    Marcel Dekker, New York, 2001.
  \bibitem{Zeilberger1999} 
    D.~Zeilberger:
    Proof of a conjecture of Chan, Robbins, and Yuen. In:
    Orthogonal polynomials: numerical and symbolic algorithms (Legan\'es, 1998),
    {\em Electron. Trans. Numer. Anal.}  {\bf 9}  (1999).
\end{thebibliography}
\end{document}